\documentclass{amsart}

\usepackage{amssymb}
\usepackage{amsthm}
\usepackage{amsmath}
\usepackage[hidelinks]{hyperref}

\theoremstyle{plain}
\newtheorem{theorem}{Theorem}[section]
\newtheorem{lemma}[theorem]{Lemma}
\newtheorem{corollary}[theorem]{Corollary}

\theoremstyle{definition}

\newtheorem{remark}[theorem]{Remark}

\numberwithin{equation}{section}

\renewcommand{\a}{\alpha}
\renewcommand{\b}{\beta}
\newcommand{\g}{\gamma}
\renewcommand{\d}{\delta}
\renewcommand{\k}{\kappa}
\renewcommand{\o}{\omega}

\renewcommand{\l}{\lambda}

\newcommand{\ps}{\emptyset}
\newcommand{\rest}{\upharpoonright}
\newcommand{\sst}{\subseteq}

\newcommand{\cu}{\mathcal U}
\newcommand{\cv}{\mathcal V}
\newcommand{\cp}{\mathcal P}

\newcommand{\F}{\mathcal F}
\newcommand{\ft}{\mathcal F^{\infty}_T}

\newcommand{\seq}[1]{\left<#1\right>}
\newcommand{\set}[1]{\left\{#1\right\}}
\newcommand{\abs}[1]{\left\vert#1\right\vert}

\newcommand{\norm}[1]{\left\lVert#1\right\rVert_{\infty}}

\newcommand{\dom}{\operatorname{dom}}

\newcommand{\supp}{\operatorname{supp}}

\newcommand{\otp}{\operatorname{otp}}
\newcommand{\h}{\operatorname{ht}}

\newcommand{\mycomment}[1]{}

\author[B. Kuzeljevi\'c]{Bori\v sa Kuzeljevi\'c}
\thanks{}
\address[Kuzeljevi\'c]{Department of Mathematics and Informatics, Faculty of Sciences, University of Novi Sad, Serbia.}
\email{borisha@dmi.uns.ac.rs}
\urladdr{\url{https://people.dmi.uns.ac.rs/\textasciitilde borisha}}
\dedicatory{}

\author[S. Milo\v sevi\'c]{Stepan Milo\v sevi\'c}
\thanks{}
\address[Milo\v sevi\'c]{Faculty of Technical Sciences, University of Novi Sad, Serbia.}
\email{stepanmilosevic@uns.ac.rs}

\author[S. Todor\v{c}evi\'{c}]{Stevo Todor\v{c}evi\'{c}}
\address[Todor\v{c}evi\'{c}]{Department of Mathematics, University of Toronto, Toronto, Canada, M5S 2E4. Institut de Math\'{e}matiques de Jussieu, UMR 7586, 2 pl. Jussieu, Case 7012, 75251 Paris Cedex 05, France. Mathematical Institute SANU, Kneza Mihaila 36, 11001 Belgrade, Serbia.}
\curraddr{}
\email{stevo@math.toronto.edu}
\email{stevo.todorcevic@imj-prg.fr}
\email{stevo@mi.sanu.ac.rs}
\urladdr{}
\dedicatory{}

\thanks{}

\title{Topological groups from matrices of sets}

\date{\today}

\begin{document}

	\begin{abstract}
		We give a general construction of topological groups from combinatorial structures such as trees, towers, gaps, and subadditive functions.
		We connect topological properties of corresponding groups with combinatorial properties of these objects.
		For example, the group built from an $\o_1$-tree is Frech\'et iff the tree is Aronszajn.
		We also determine cofinal types of some of these groups under certain set theoretic assumptions.
	\end{abstract}

\maketitle

\section{Introduction}

The purpose of this note is to provide a general construction of new examples of topological groups.
The key feature of groups we will present is that they are all built based on well known set-theoretic objects such as subadditive functions, trees, gaps, and towers.
For all undefined notions, we refer the reader to Section \ref{s:preliminaries}.
This research direction originated in the work by the third author, note the most recent contribution in \cite[Theorem 4.9]{pid_fund} and results in \cite{walks}, and extended by the second and the third author in \cite{mt_walks}.
There, topological groups were constructed from characteristics of walks on ordinals.
There were related constructions preceeding these, the reader can take a look at \cite{dw,gt,cs,szept} and a historical remark after Theorem 4.9 in \cite{pid_fund}.
Here we generalize those constructions and find new examples.
In our framework, we are able to characterize Frech\'etness of analyzed groups via combinatorial properties of underlying objects.
For example, by Theorem \ref{t:treefrechet} the group we get from an $\o_1$-tree is Frech\'et iff the tree is Aronszajn.

There are many problems in the literature involving Frech\'et topological groups.
To mention just a few, recall that the solution to Malykhin's problem has been completed recently in \cite{malykhin}, see also \cite{hrusakshibakov}.
It asks whether there is a countable Frech\'et non-metrizable group, see \cite[Question 26]{arh}.
The solution is that, although there are many consistent examples, there is a model where no such group exists.
There is also a well-known, still open, problem whether in ${\rm ZFC}$ there is a Frech\'et group whose square is not Frech\'et.
See \cite[pp. 154]{arh} for a discussion about this class of problems or \cite[pp. 474]{shakhmatov} where the problem is asked as mentioned.
The furthest advancement on this problem is in \cite{stevoSL} where two Frech\'et groups are constructed whose product is not countably tight.
Another problem we find interesting was posed by Feng in \cite{feng}: Is there, in ${\rm ZFC}$, a non-metrizable Frech\'et group $G$ such that $[\o_1]^{<\o}\not\le_T G$.
One consistent example is $G_{\rho_{2}}(\theta)$ for $\theta>\o_1$ using $\square(\theta)$ from \cite[Theorem 4.9]{pid_fund}, while the other is $G_{\rho_2}(\omega_1)$ using Ostaszewski's club principle in \cite{mt_walks}.
This question relates the problem of metrizability of a group to its Tukey type.
This is expected as in this language, Birkhoff-Kakutani theorem says that a group $G$ is metrizable if and only if $G\le_T \o$.
It should also be mentioned that the third author in \cite{stbasic} has found another equivalent condition for metrizability in the class of Frech\'et groups.
He showed that a Frech\'et group is metrizable if and only if it is Tukey reducible to a basic order in the sense of \cite{basic}.

Tukey types of topological groups and spaces have been recently investigated in papers by Gartside, Feng, Morgan, and Mamatelashvili, the reader can take a look at \cite{gartfeng,gartmam,gartmamsubsets,gartmor,morgan}.
Some of their work is formulated in terms of calibres.
Usually this can be rephrased in the language of Tukey reductions - for example, $D$ has calibre $(\o_1,\o)$ if and only if $[\o_1]^{<\o}\not\le_T D$.
The recent paper \cite{rejection} is relevant for the general Tukey theory of topological groups as, among other results, they show that for any directed set $D$ there is a topological group $G$ such that $G\equiv_T D$.
Significant related theory, although presented in a different language, can be found in an extensive paper of Banakh \cite{banakh}.
For this line of research, one can take a look at \cite{gabkaklei}, \cite{banschu}, or \cite{dowfeng} and \cite{feng}.
Already mentioned \cite{stbasic} also provides a more general view on that topic.

\subsection{The outline of the paper}
In Section \ref{s:general}, we describe our general framework for building topological groups from $(\k,\l)$-matrices where $\l\le\k$ are regular infinite cardinals.
The parameter $\k$ refers to the cardinality and character of the group, which are both $\k^+$, while the parameter $\l$ refers to the tightness of the group.
In Section \ref{s:tightness} we show that the group obtained from a strong matrix is Frech\'et iff the matrix is unbounded in a certain well-defined sense.
In Section \ref{s:subadditive} we show that each transitive function gives a matrix, while subadditive one gives a strong matrix.
As a corollary we get groups from metrics on ordinals and from $\k^+$-towers in $\cp(\k)$.
In Section \ref{s:tree} we present two constructions of topological groups from trees.
The first is very general, and gives the mentioned Theorem \ref{t:treefrechet}: the group is Frech\'et iff the tree is Aronszajn.
The other is for a specific class of trees, and allows one to get the non-metrizable Frech\'et group $G(\ft)$ which does not have $[\o_1]^{<\o}$ as a Tukey quotient.
This gives another consistent counterexample for Feng's question, built from a particular Souslin tree.
In the last section, Section \ref{s:gaps}, we give a construction of the topological group from any $(\k^+,\k^+)$ pre-gap in $\cp(\k)$.
This is different in a sense that we do not get exactly a matrix, but we show that the obtained family is good enough to give the subbase for our topological group.
Then we prove that in order for our group to be countably tight, the pre-gap in fact has to be a gap.

\section{Preliminaries}\label{s:preliminaries}

We use mostly standard notation.
When set theory is concerned, standard reference with all the relevant background is \cite{jech}.
If $A$ is a set, then its cardinality is denoted $\abs{A}$.
For a set $A$ and a cardinal $\k$, we denote $[A]^{\k}=\set{X\sst A:\abs{X}=\k}$.
Sets $[A]^{<\k}$ and $[A]^{\le\k}$ are defined similarly.
For a set $A$, we denote its powerset by $\mathcal P(A)$.
For sets $A$ and $B$, function $f:A\to B$, and $X\sst A$ we denote $f[X]=\set{f(x):x\in X}$.
For $Y\sst B$ we denote $f^{-1}[Y]=\set{a\in A:f(a)\in Y}$.
For sets of ordinals $a$ and $b$ we write $a<b$ when $\a<\b$ for all $\a\in a$ and $\b\in b$.
All variations on this notion are handled similarly.
If $\theta$ is a regular cardinal, then $H(\theta)$ is the set of all sets of hereditary cardinality less than $\theta$.
In this case, $H(\theta)$ satisfies all axioms of $\rm{ZFC}$ except the powerset axiom.
If $\k$ is a cardinal, $C\sst \k$ is a \emph{club} in $\k$ if it is closed and unbounded in $\k$.
A set $S\sst \k$ is \emph{stationary} if it intersects every club in $\k$.
Every club in $\k$ is stationary in $\k$.
The following lemma is well known, and will be used later in the paper.

\begin{lemma}\label{l:elsubmodel}
	Let $\theta$ be an uncountable regular cardinal, $\k$ a regular cardinal such that $\k^+<\theta$, $S$ stationary in $\k^+$, and $X\in H(\theta)$ of size at most $\k$.
	Then, there is an elementary submodel $M$ of $H(\theta)$ such that $\abs{M}=\k$, $X\sst M$, and $\k^+\cap M\in S$.
\end{lemma}

For Tukey theory, a good reference for standard notions is \cite{stevodirected}.
A poset $(D,\le_D)$ is \emph{directed set} if for all $x$ and $y$ in $D$, there is $z$ in $D$ such that $x\le_D z$ and $y\le_D z$.
For example, in a topological space, a local base of any point is directed by the relation $\supseteq$.
If $(D,\le_D)$ is a directed set, then $X\sst D$ is \emph{bounded} if there is $d\in D$ such that $x\le_D d$ for all $x\in X$, otherwise it is \emph{unbounded}.
A set $Y\sst D$ is \emph{cofinal} in $D$ if for each $d\in D$ there is $y\in Y$ such that $d\le_D y$.
For directed sets $(D,\le_D)$ and $(E,\le_E)$, we say that $f:D\to E$ is a \emph{Tukey map} if $f[X]$ is unbounded in $E$ for any unbounded $X$ in $D$.
This is the same as saying that the preimage of any bounded set is bounded.
When there is a Tukey map $f:D\to E$, we say that $D$ is Tukey reducible to $E$ (or that $D$ is a Tukey quotient of $E$) and write $D\le_T E$.
It is well-known that $D\le_T E$ iff there is convergent map from $E$ to $D$ (see \cite{stevodirected} or \cite{schmidt}).
Recall that $f:E\to D$ is a \emph{convergent} (cofinal) map if for any $d\in D$ there is $e\in E$ such that $f(x)\ge_D d$ for all $x\ge_E e$.
This is the same as saying that the image of any cofinal set is cofinal.
Tukey proved that $D\le_T E$ and $E\le_T D$ (in short $D\equiv_T E$) iff both $D$ and $E$ can be embedded as cofinal subsets into a single poset (see \cite{tukey} or \cite{stevodirected}).
We will often use the following remark.

\begin{remark}\label{r:types}
	If $\k$ is a regular infinite cardinal and $A\sst [\k]^{<\o}$ is of cardinality $\k$ and directed with respect to $\sst$, then $\left([\k]^{<\o},\sst\right)\equiv_T (A,\sst)$.
\end{remark}

For topological groups, standard reference is \cite{knjiga}.
A \emph{topological group} is an abstract group $(G,\cdot)$ equipped with a topology such that operations $\cdot:G\times G\to G$ and ${}^{-1}:G\to G$ are continuous.
For an abstract group $G$, we typically denote its identity by $e$.
Tukey theory in the context of topological groups has been present for a long time, but for basic facts and definitions we refer to \cite{km}. In particular, if $\mathcal B_e$ is the collection of all open neighborhoods of the identity $e$ in $G$ and $(D,\le_D)$ is a directed set, we say that $G\le_T D$ if $(\mathcal B_e,\supseteq)\le_T (D,\le_D)$.
We define $G\le_T H$ similarly for groups $G$ and $H$.
We will mostly use the following fact for building topological groups, see \cite[Theorem 1.3.12]{knjiga} or \cite[Theorem 4.5]{hr}.

\begin{theorem}\label{t:osnovna}
	Let $G$ be a group with the identity $e$, and $\cu$ a family of subsets of $G$ satisfying the following conditions:
	\begin{enumerate}
		\item for every $U\in\cu$ there is an element $V\in\cu$ such that $V^2\sst U$;
		\item for every $U\in\cu$ there is an element $V\in\cu$ such that $V^{-1}\sst U$;
		\item for every $U\in\cu$ and every $x\in U$, there is $V\in\cu$ such that $Vx\sst U$;
		\item for every $U\in\cu$ and $x\in G$, there is $V\in\cu$ such that $xVx^{-1}\sst U$;
		\item for $U,V\in\cu$, there is $W\in\cu$ such that $W\sst U\cap V$;
		\item $\set{e}=\bigcap \cu$.
	\end{enumerate}
	Then $G$ is a topological group with the topology generated by the base $\set{Ua:a\in G}$, i.e. $\cu$ is an open base of the identity in $G$.
\end{theorem}

The notation $\overline{A}$ will be reserved for the closure of $A$ in the appropriate topology.

A topological group $G$ is $\emph{countably tight}$ if for all $A\sst G$ with $e\in\overline{A}$, there is a countable $X\sst A$ such that $e\in \overline{X}$.
For a cardinal $\l$, we say that $G$ has \emph{tightness at most $\l$} if for every $A\sst G$ with $e\in \overline{A}$, there is $X\in [A]^{\le\l}$ such that $e\in\overline{X}$.

A topological group $G$ is \emph{Frech\'et} if for all $A\sst G$ with $e\in \overline{A}$, there is a sequence $s$ in $A$ such that $\lim s=e$.
Recall that $G$ is an \emph{$(\a_1)$-group} if for any countable collection of sequences $\set{s_n:n<\o}$ in $G$ such that $\lim s_n=e$ for each $n<\o$, there is a sequence $s$ in $G$ such that $\lim s=e$ and that $s\setminus s_n$ is finite for all $n<\o$.
The following lemma is well known.

\begin{lemma}\label{l:alpha_1}
	If $G$ is a topological group of cardinality $\aleph_1$ which is an increasing union of metrizable subsets, then $G$ is an $(\a_1)$-group.
\end{lemma}

\section{Matrices}\label{s:general}

In this section we describe a general method for constructing topological groups on the set $[\k^+]^{<\o}$ for any infinite regular cardinal $\k$.

Let $A$ and $B$ be sets of ordinals.
We say that a family $$\mathcal F=\set{F_{\xi}(\a): \a\in A,\ \xi\in B}$$ is an \emph{$(A,B)$-matrix} if the following conditions hold:
\begin{itemize}
	\item[(G1)] $\bigcup_{\xi\in B}F_{\xi}(\a)=\a$ for every $\a\in A$,
	\item[(G2)] $F_{\xi}(\a)\sst F_{\eta}(\a)$ for each $\a\in C$ and all $\xi\le\eta$ in $B$,
	\item[(G3)] for all $\a\le\b$ in $A$ and each $\xi$ in $B$, there is $\eta\in B$ so that $F_{\xi}(\a)\sst F_{\eta}(\b)$.
\end{itemize}

If $\l\le\k$ are infinite regular cardinals, $\mathcal F$ is a $(C,\l)$-matrix, and $C$ is a cofinal subset of $\k^+$, then we will say $\mathcal F$ is a $(\k,\l)$-matrix indexed by $C$.
If $C$ is clear from the context, or not essential, we will just say $(\k,\l)$-matrix.
Finally, if $\k=\l$, we will simply say that $\mathcal F$ is a $\k$-matrix.

For the remaining of this section fix $\l\le\k$ regular infinite cardinals, and a $(\k,\l)$-matrix $\mathcal F$ indexed by $C$. 
For $\xi<\l$ and $\a\in C$ we set $$U_{\xi}(\a)=\set{a\in [\k^+]^{<\o}: a\cap F_{\xi}(\a)=\ps}.$$
Then, to a matrix $\mathcal F$ we adjoin a family $\cu(\mathcal F)=\set{U_{\xi}(\a):\xi<\l,\ \a\in C}$. 

In \cite{pid_fund} and \cite{mt_walks} two different $(\k,\l)$-matrices were introduced, depending on characteristics of walks $\rho_1$ and $\rho_2$, and it was proved that they give topological groups.
Recall that $\triangle$ is the operation of symmetric difference.

\begin{theorem}\label{t:tgroup}
	Suppose that $\l\le\k$ are regular infinite cardinals.
	Let $$\mathcal F=\set{F_{\xi}(\a):\a\in C,\ \xi<\l}$$ be a $(\k,\l)$-matrix.
	Then $\cu(\mathcal F)$ is the neighborhood basis of the identity for a topology on $[\k^+]^{<\o}$ which makes $([\k^+]^{<\o},\triangle)$ a topological group.
\end{theorem}

\begin{proof}
	It is sufficient to show that $\cu(\mathcal F)$ satisfies conditions (1)-(6) in Theorem \ref{t:osnovna}.

    To see that (1) and (3) are true, take $\xi<\k$ and $\a\in C$.
	Note that we have  $U_{\xi}(\a)\triangle U_{\xi}(\a)\sst U_{\xi}(\a)$ because $a\triangle b\sst a\cup b$, and consequently $a\triangle b\in U_{\xi}(\a)$ for $a,b\in U_{\xi}(\a)$.
    For (2) just observe that $a\triangle a=\ps$ for $a\in [\k^+]^{<\o}$, so $U^{-1}=U$ for each $U\in\cu(\mathcal F)$. Condition (4) holds by the commutativity of $\triangle$.
    
    Now we prove item (5).
    Take any $U_{\xi}(\a)$ and $U_{\eta}(\b)$ in $\cu(\mathcal F)$.
    Assume $\a\le\b$.
    Since $\mathcal F$ is a $(\k,\l)$-matrix, by (G3), there is some $\xi_0$ such that $F_{\xi}(\a)\sst F_{\xi_0}(\b)$.
    Now take $\zeta=\max\set{\eta,\xi_0}$.
    Then, using (G2), we have $F_{\xi}(\a)\cup F_{\eta}(\b)\sst F_{\zeta}(\b)$, so $U_{\zeta}(\b)\sst U_{\xi}(\a)\cap U_{\eta}(\b)$ as required.

    To see (6), suppose there is $a\in [\k^+]^{<\o}\setminus\set{\emptyset}$ so that $a\in \bigcap \cu(\mathcal F)$.
    Let $\a\in C$ be such that $a\sst \a$.
	This is possible as $C$ is unbounded in $\k^+$.
    Take $x\in a$.
    Then, by (G1), there is $\xi<\l$ so that $x\in F_{\xi}(\a)$.
	This means that $a\cap F_{\xi}(\a)\neq\ps$, i.e. $a\notin U_{\xi}(\a)$, contradicting the assumption that $a$ is a non-empty set in $\bigcap\cu(\mathcal F)$.
\end{proof}

\begin{remark}\label{r:directed}
	If $\mathcal F$ is a $(\k,\l)$-matrix, then $(\mathcal F,\sst)$ is a directed set.
	This follows from property (G3), and also can be seen as a consequence of Theorem \ref{t:tgroup} or its proof (the proof of (5) from Theorem \ref{t:osnovna}).
\end{remark}

For the remaining of the paper, if $\mathcal F$ is a $(\k,\l)$-matrix, then $G(\mathcal F)$ will always denote the topological group given by Theorem \ref{t:tgroup}.
Let us first present some properties of these groups.

\begin{theorem}\label{t:character}
	Suppose that $\l\le\k$ are infinite regular cardinals.
	Let $\mathcal F$ be a $(\k,\l)$-matrix.
	Then the group $G(\mathcal F)$ has character $\k^+$.
\end{theorem}
\begin{proof}
	It is clear that $\cu(\mathcal F)$ has cardinality $\k^+$, so it is enough to prove that there is no local base of $\ps$ in $G(\mathcal F)$ of cardinality less then $\k^+$.
    Suppose that $\mu\le\k$ and that $\set{W_{\g}:\g<\mu}$ is a local base of $\ps$ in $G(\mathcal F)$.
    For each $\g<\mu$ take $U_{\xi_{\g}}(\a_{\g})\sst W_{\g}$.
    This is possible as $\cu(\mathcal F)$ is a local base of $\ps$ in $G(\mathcal F)$.
	Now take ordinals $\b'$ and $\b$ so that $\b$ is in $C$ and that $$\sup\set{\a_{\g}:\g<\mu}<\b'<\b<\k^+.$$
	Note that $\b'\notin F_{\xi_{\g}}(\a_{\g})$ for any $\g<\mu$.
	This implies that $\set{\b'}\in U_{\xi_{\g}}(\a_{\g})$ for $\g<\mu$.
	Since $\b=\bigcup_{\xi<\l}F_{\xi}(\b)$ by (G1), there is some $\xi_0$ such that $\b'\in F_{\xi_0}(\b)$.
    Consider $U_{\xi_0}(\b)$, clearly $\set{\b'}\notin U_{\xi_0}(\b)$.
    Let $\d<\mu$ be such that $W_{\d}\sst U_{\xi_0}(\b)$.
    Observe that in that case $U_{\xi_{\d}}(\a_{\d})\sst U_{\xi_0}(\b)$.
    But we already explained that $\set{\b'}\in U_{\xi_{\d}}(\a_{\d})$ and $\set{\b'}\notin U_{\xi_0}(\b)$, contradicting the last observation.
\end{proof}

\begin{remark}\label{r:tukey}
	If $\mathcal F$ is a $(\k,\l)$-matrix, then $(\cu(\mathcal F),\supseteq)$ is isomorphic to $(\mathcal F,\sst)$, so $$G(\mathcal F)\equiv_T \mathcal F.$$
\end{remark}

\begin{theorem}\label{t:smalltype}
	Let $\l\le\k$ be regular infinite cardinals and $\mathcal F$ a $(\k,\l)$-matrix.
	Then $1$ and $\l$ are the only Tukey quotients of the group $G(\mathcal F)$ of cardinality less than $\k^+$.
\end{theorem}

\begin{proof}
	By Remark \ref{r:tukey}, it is enough to show that if $D\le_T (\mathcal F,\sst)$ and $\abs{D}\le\k$, then $D\equiv_T 1$ or $D\equiv_T \l$.
	So take such $D$ and let $f:D\to \mathcal F$ be a Tukey map.
	Then for each $d \in D$ there are $\xi_d < \l$ and $\a_d \in C$ such that $f(d) = F_{\xi_d}(\a_d)$.
	By the cardinality assumption on $D$, there is some $\g\in C\setminus \left(\sup\set{\a_d:d\in D}+1\right)$.
	Since $\mathcal F$ is a $(\k,\l)$-matrix, by (G3), for each $d\in D$, there is some $\eta_d<\l$ such that $F_{\xi_d}(\a_d)\sst F_{\eta_d}(\g)$.
	Now for each $\xi<\l$ let $$D_{\xi}=\set{d\in D:f(d)\sst F_{\xi}(\g)}.$$
	Then by the choice of $\g$ and the choice of the set $\set{\eta_d:d\in D}\sst\l$, we have $$\textstyle D=\bigcup_{\xi<\l}D_{\xi}.$$
	Note that $D_{\xi}$ is bounded in $D$ for all $\xi<\l$.
	Because if not, since $f$ is Tukey, $f[D_{\xi}]$ would be unbounded in $\mathcal F$ which is clearly not true (it is bounded by $F_{\xi}(\g)$).
	Let $d_{\xi}$ bound $D_{\xi}$ for each $\xi<\l$, i.e. $d\le d_{\xi}$ for each $\xi<\l$ and $d\in D_{\xi}$.
	
	Define $g:\l\to D$ by $g(\xi)=d_{\xi}$.
	We prove that $g$ is a cofinal map.
	Take any $d\in D$.
	Let $\xi$ be such that $d\in D_{\xi}$.
	Let $\eta\ge\xi$ be arbitrary.
	As $F_{\xi}(\g)\sst F_{\eta}(\g)$ by (G2), we know that $D_{\xi}\sst D_{\eta}$, so $d\le d_{\eta}$.
	So for any $\eta\ge \xi$, we have $g(\eta)\ge d$, which means that $g$ is a cofinal map.
    Hence $D\le_T \l$, so either $D\equiv_T 1$ or $D\equiv_T \l$.
\end{proof}

We say that a $(\k,\l)$-matrix is a \emph{strong $(\k,\l)$-matrix} if it satisfies additionally:
\begin{itemize}
	\item[(G4)] for all $\a\le\b$ in $C$ and each $\eta<\l$ there is $\xi<\l$ so that $F_{\eta}(\b)\cap \a\sst F_{\xi}(\a)$.
\end{itemize}
We can show that groups constructed from strong matrices have nicer properties.
For a $(\k,\l)$-matrix $\mathcal F$ and $\d\in C$, considering the introduced notation, we denote: $$\cu_{\d}(\mathcal F)=\set{[\d]^{<\o}\cap U_{\eta}(\d):\eta<\l}.$$

\begin{theorem}\label{t:smallcharacter}
    Suppose $\l\le\k$ are infinite regular cardinals and $\mathcal F$ is a strong $(\k,\l)$-matrix.
	Then for $\d\in C$, the family $\cu_{\d}(\mathcal F)$ is a local base of $\ps$ in the group $([\d]^{<\o}, \triangle)$ with the topology induced from $G(\mathcal F)$.
\end{theorem}

\begin{proof}
    Fix $\d\in C$.
    To prove that the family $\cv_{\d}$ is a local base of the identity in the given group it is enough to show that for each $\xi<\l$ and $\a\in C$ there is some $\eta<\l$ such that $$[\d]^{<\o}\cap U_{\eta}(\d)\sst [\d]^{<\o}\cap U_{\xi}(\a).$$
    To show this, it is enough to prove that $F_{\xi}(\a)\cap \d\sst F_{\eta}(\d)\cap \d$ for a suitably chosen $\eta$.
    Let us consider two cases.
    First, that $\a\le\d$.
    In this case, by (G3), there is $\eta<\l$ such that $F_{\xi}(\a)\sst F_{\eta}(\d)$, which clearly implies the required inclusion.
    Second, that $\d<\a$.
    In this case, by (G4), there is $\eta<\l$ such that $F_{\xi}(\a)\cap\d\sst F_{\eta}(\d)=F_{\eta}(\d)\cap\d$, as required.
\end{proof}

\begin{corollary}\label{c:nonmetrizable}
	If $\mathcal F$ is an $\o$-matrix, then $G(\mathcal F)$ is not metrizable.
\end{corollary}
\begin{proof}
	By Theorem \ref{t:character} and the Birkhoff-Kakutani theorem.
\end{proof}

\begin{corollary}\label{c:alpha_1}
	If $\mathcal F=\set{F_n(\a):n<\o,\ \a\in C}$ is a strong $\o$-matrix and $\d\in C$, then the group $([\d]^{<\o},\triangle)$ with the topology induced from $G(\mathcal F)$ is metrizable.
	In particular, $G(\mathcal F)$ is a non-metrizable $(\a_1)$-group.
\end{corollary}
\begin{proof}
	By Theorem \ref{t:smallcharacter} and Lemma \ref{l:alpha_1}.
\end{proof}

\begin{theorem}\label{t:maxtype}
	Suppose that $\k$ is a regular infinite cardinal.
	Let $\mathcal F\sst [\k^+]^{<\o}$ be a $(\k,\o)$-matrix.
	Then $[\k^+]^{<\o}\le_T G(\mathcal F)$.
\end{theorem}
\begin{proof}
	Directly by Remark \ref{r:tukey} and Remark \ref{r:types}.
\end{proof}

\section{Tightness and unboundedness}\label{s:tightness}

For infinite regular cardinals $\k$ and $\l$, we say that a function $f:[\k]^2\to \l$ is \emph{$\k$-unbounded} if for all $\xi<\l$ and every family $\mathcal A$ of size $\k$ consisting of pairwise disjoint finite subsets of $\k$, there are different $a$ and $b$ in $\mathcal A$ such that $f(\a,\b)>\xi$ for all $\a\in a$ and $\b\in b$ with $\a<\b$.

We proceed to explain how to every $(\k,\l)$-matrix indexed by $\k^+$ we can adjoin a function in a way that unboundednes of the function relates to the tightness of the topological group corresponding to the matrix.

Suppose that $\l\le\k$ are infinite regular cardinals and that we are given a $(\k,\l)$-matrix $\mathcal F=\set{F_{\xi}(\a):\xi<\l,\a<\k^+}$.
Define the function $\rho_{\mathcal F}:[\k^+]^2\to \l$ as follows: $$\rho_{\mathcal F}(\a,\b)=\min\set{\xi<\l: \a\in F_{\xi}(\b)}$$ for all $\a<\b<\k^+$.
The only reason to define $\rho_{\mathcal F}$ only for matrices indexed by $\k^+$ is that for our applications $\rho_{\mathcal F}$ should be defined for all $\a$ and $\b$ in $\k^+$, and there is no meaningful way to do it if $F_{\xi}(\b)$ does not exist for all $\b<\k^+$.

\begin{remark}\label{r:funcmatr}
	Note that with the notation as in the previous paragraph, for $\eta<\l$ and $\b<\k^+$, we have $F_{\eta}(\b)=\set{\a<\b:\rho_{\mathcal F}(\a,\b)\le \eta}$.
	We prove this.
	For the inclusion $\sst$, take some $\a\in F_{\eta}(\b)$.
	Suppose the contrary, that $\rho_{\mathcal F}(\a,\b)>\eta$.
	Then $\min(A)>\eta$ where $A=\set{\xi<\l:\a\in F_{\xi}(\b)}$.
	This is clearly impossible as $\eta\in A$.
	For the other inclusion $\supseteq$, take some $\a<\b$ such that $\rho_{\mathcal F}(\a,\b)\le\eta$.
	Then $\a\in F_{\xi}(\b)$ for some $\xi\le\eta$, but by (G2), then we have $\a\in F_{\xi}(\beta)\sst F_{\eta}(\b)$, as required.
\end{remark}

We will say that a $(\k,\l)$-matrix $\F$ is \emph{unbounded} if $\rho_{\mathcal F}$ is a $\k^+$-unbounded function.
In order to simplify the statements of the next theorem and its corollaries, we introduce another definition.
We say that a topological group $G$ is $\l$-Frech\'et, for a regular infinite cardinal $\l$, if for any set $A\sst G$ with $e\in\overline{A}$ there is a $\l$-sequence $\seq{x_{\xi}:\xi<\l}$ in $A$ such that for every open set $U$ containing $e$, there is some $\eta<\l$ so that $x_{\xi}\in U$ whenever $\eta\le\xi<\l$.
Clearly, $\o$-Frech\'et is the same as Frech\'et.

\begin{theorem}\label{t:frechet}
	If $\l\le\k$ are regular infinite cardinals, and $\mathcal F$ is a strong $(\k,\l)$-matrix indexed by $\k^+$, then 
	$G(\mathcal F)$ is $\l$-Frech\'et if and only if $\mathcal F$ is unbounded.
\end{theorem}
\begin{proof}
	By Theorem \ref{t:osnovna}, we know that $G(\mathcal F)$ is a topological group.
	We first assume that the matrix $\mathcal F$ is unbounded.
	To show that $G(\mathcal F)$ is $\l$-Frech\'et, take $A\sst G(\mathcal F)$ such that $\ps$ is in the closure of $A$.
	Let $\theta$ be a regular cardinal above $\left(2^{\o_1}\right)^+$.
	Take $M$, an elementary submodel of $(H(\theta),\in)$ of size $\k$, containing $\mathcal F$ and $A$ as elements.
	Denote $\d=M\cap \k^+$.
	Note that $\d$ is an ordinal in $\k^+$ and that $\l\sst\d\sst M$.
	To prove this direction of the theorem, it is enough to show that $\ps$ is in the closure of $M\cap A$.
	To see this, note that by Theorem \ref{t:smallcharacter}, we know that for every basic open neighborhood of $\ps$ in $G(\mathcal F)$, there is some $\xi<\l$ such that the induced open set on $[\d]^{<\o}$ is exactly $[\d]^{<\o}\cap U_{\xi}(\d)$.
	If $\ps$ were in the closure of $A\cap M$, we would be able to pick $a_{\xi}\in U_{\xi}(\d)\cap A\cap M$ for each $\xi<\l$.
	Then the $\l$-sequence $\seq{a_{\xi}:\xi<\l}$ would be in $A$.
	Moreover, for any open $U_{\eta}(\d)$ and each $\xi>\eta$ we would have $a_{\xi}\in U_{\xi}(\d)\sst U_{\eta}(\d)$ (this holds as $F_{\eta}(\d)\sst F_{\xi}(\d)$ by the definition of matrix).
	Thus it would prove that $G(\mathcal F)$ is $\l$-Frech\'et.

	So we continue to prove $\ps\in \overline{A\cap M}$.
	If $\abs{A}<\k^+$, then $A\sst M$ and by the above observation the theorem would be proved, so assume that $A$ is of cardinality $\k^+$.
	We now work towards proving $\ps\in\overline{A\cap M}$.
	So take $\xi<\l$.
	Note that $\xi\in M$.
	We will prove that $U_{\xi}(\d)\cap A\cap M\neq\ps$, i.e. that there is some $a\in A\cap M$ such that $a\cap F_{\xi}(\d)=\ps$.
	Since $\ps\in \overline{A}$, there is some $b\in A$ such that $b\cap F_{\xi}(\d)=\ps$.
	We can assume the $b\notin M$ (otherwise the proof is finished).
	Denote $c=b\setminus M$.
	
	Consider the formula $$\psi(\a)\equiv \left(\exists w\in [\k^+]^{<\o}\right)\ \a<w\ \wedge\ w\cup (b\cap \d)\in A.$$
	To see that $M\models \big((\forall \a<\k^+)\psi(\a)\big)^M$, take arbitrary $\a\in M\cap \k^+=\d$ for a moment.
	Then $H(\theta)\models \psi(\a)$, namely $b\setminus \d$ is a witness for this.
	Since $A,[\k^+]^{<\o},b\cap\d$, and $\a$ are all elements of $M$, by elementarity, $\psi(\a)$ holds in $M$.
	Since $\a$ was arbitrary in $M\cap \k^+$, we proved $M\models \big((\forall \a<\k^+)\psi(\a)\big)^M$.
	Again, by elementarity, $H(\theta)\models (\forall \a<\k^+)\psi(\a)$.
	This means that, in $H(\theta)$, for each $\a<\k^+$ there is some $w_{\a}$ such that $\a\cap w_{\a}=\ps$ and $w_{\a}\cup (b\cap\d)\in A$.
	Moreover, we can choose $w_{\a}$ to be minimal with respect to the anti-lexicographical ordering of $[\k^+]^{<\o}$.
	As this ordering is an element of $M$, it means that the function $f:\k^+\to [\k^+]^{<\o}$ given by $f(\a)=w_{\a}$ is definable in $M$ by parameters in $M$, so $f\in M$.
	Hence, the family $\mathcal A=\set{w_{\a}:\a<\k^+}$ is an element of $M$.

	To finish the proof, since $b\cap F_{\xi}(\d)=\ps$, we need to show that there is some $\a<\d$ such that $w_{\a}\cap F_{\xi}(\d)=\ps$.
	Equivalently, we need $\a<\d$ such that $\rho_{\mathcal F}(x,\d)>\xi$ for each $x\in w_{\a}$.
	Suppose this is not the case, i.e. $$(\forall \a<\d)(\exists x\in w_{\a})\ \rho_{\mathcal F}(x,\d)\le\xi.$$
	Consider the family $\mathcal C=\set{\b<\k^+:(\forall \a<\b)(\exists x\in w_{\a})\rho_{\mathcal F}(x,\b)\le\xi}$.
	Since $\mathcal F\in M$, we know that $\rho_{\mathcal F}\in M$.
	So, as $\xi\in M$ as well, we have $\mathcal C\in M$.
	From $\d\in \mathcal C\setminus M$ (by the assumption), we know that $\mathcal C$ is cofinal in $\k^+$. 
	Hence, for each $\a<\k^+$ there is some minimal $\g_{\a}\ge\a$ such that $\rho_{\mathcal F}(x,\g_{\a})\le \xi$ for some $x\in w_{\a}$.
	Now consider the family $\mathcal A'=\set{w_{\a}\cup\set{\g_{\a}}:\a<\k^+}$.
	This is an uncountable family of pairwise disjoint finite subsets of $\k^+$.
	Since $\rho_{\mathcal F}$ is $\k^+$-unbounded (the assumption of the theorem), there are $\a<\b<\k^+$ such that $\rho_{\mathcal F}(x,y)>\xi$ for each $x\in w_{\a}\cup\set{\g_{\a}}$ and $y\in w_{\b}\cup\set{\g_{\b}}$.
	But, this in particular means that there is no $x\in w_{\a}$ such that $\rho_{\mathcal F}(x,\g_{\b})\le \xi$, contradicting the choice of $\g_{\b}$.
	Hence, there is some $\a_0<\d$ such that $\rho_{\mathcal F}(x,\d)>\xi$ for each $x\in w_{\a_0}$.
	Then, since $\a_0<\d$, we know that $w_{\a_0}\in M\cap \mathcal A$ and $w_{\a_0}\cap F_{\xi}(\d)=\ps$.
	Consequently $(b\cap\d)\cup w_{\a_0}\in A\cap M$ and $\left((b\cap\d)\cup w_{\a_0}\right)\cap F_{\xi}(\d)=\ps$, as required.
	So, the proof of one direction of the theorem is completed.

	For the other direction, let us assume that $G(\mathcal F)$ is $\l$-Frech\'et.
	Suppose that $\rho_{\mathcal F}$ is not $\k^+$-unbounded.
	This means that there are $\xi_0<\l$ and a family $\mathcal A$ of size $\k^+$ consisting of pairwise disjoint finite subsets of $\k^+$ such that for all different $a$ and $b$ in $\mathcal A$ there are $\a\in a$ and $\b\in b$ such that $\rho_{\mathcal F}(\a,\b)\le \xi_0$.
	Then $\mathcal A\sst G(\mathcal F)$.
	Let us show that $\ps\in\overline{\mathcal A}$.
	Take any basic open $U_{\eta}(\b)$.
	Since $\mathcal A$ is a family of pairwise disjoint finite sets of size $\k^+$, there is some $a\in \mathcal A$ with $\b<a$.
	Then $a\cap F_{\eta}(\b)=\ps$, i.e. $a\in \mathcal A\cap U_{\eta}(\b)$.
	So, since $G(\mathcal F)$ is $\l$-Frech\'et, there is a $\l$-sequence $\seq{a_{\xi}:\xi<\l}$ in $\mathcal A$ such that: \begin{equation}\label{eq:convergence}\mbox{for any open }U_{\zeta}(\g)\mbox{ there is }\eta<\l\mbox{ so that }a_{\xi}\in U_{\zeta}(\g)\mbox{ whenever }\eta\le\xi<\l.\end{equation}
	Take $a\in \mathcal A$ such that $a_{\xi}<a$ for each $\xi<\l$.
	This is possible as $\mathcal A$ is of size $\k^+$ and $\l\le\k$.
	Enumerate $a=\set{\b_0,\dots,\b_{l-1}}$ for some $l<\o$.
	By the assumption on $\mathcal A$, for each $\xi<\l$ there is some $i<l$ and $\a\in a_{\xi}$ such that $\rho_{\mathcal F}(\a,\b_i)\le \xi_0$.
	Hence, there is $j<l$ and a set $J$ cofinal in $\l$, so that for $\xi\in J$ there is $\a\in a_{\xi}$ such that $\rho_{\mathcal F}(\a,\b_j)\le \xi_0$.
	But this implies that $a_{\xi}\notin U_{\xi_0}(\b_j)$ for each $\xi\in J$, contradicting the assumption that $\seq{a_{\xi}:\xi<\l}$ satisfies (\ref{eq:convergence}) for $\zeta=\xi_0$ and $\g=\b_j$.
\end{proof}

We have two immediate corollaries.

\begin{corollary}
	If $\k$ is a regular infinite cardinal and $\mathcal F$ is a strong $(\k,\o)$-matrix indexed by $\k^+$, then $G(\mathcal F)$ is Frech\'et if and only if it is unbounded.
\end{corollary}

\begin{corollary}
	If $\l\le\k$ are regular infinite cardinals and $\mathcal F$ is an unbounded strong $(\k,\l)$-matrix indexed by $\k^+$, then $G(\mathcal F)$ has tightness at most $\l$.
\end{corollary}

Note that more properties of the group $G(\mathcal F)$ can be read from the function $\rho_{\mathcal F}$.
Suppose that we are given infinite regular cardinals $\theta$ and $\l$, and a function $f:[\theta]^2\to\l$. We say that the function $f$ satisfies condition (H) if:
\begin{itemize}
	\item[(H)] the set $\set{\a<\b:f(\a,\b)\le\xi}$ is finite for all $\b<\theta$ and $\xi<\l$.
\end{itemize}
Then, by Theorem \ref{t:maxtype} and Remark \ref{r:funcmatr}, we have:

\begin{corollary}\label{c:maxtypeH}
	If $\l\le\k$ are infinite regular cardinals, $\mathcal F$ is a strong $(\k,\l)$-matrix indexed by $\k^+$, and $\rho_{\mathcal F}$ satisfies the condition (H), then $G(\mathcal F)\equiv_T [\k^+]^{<\o}$. 
\end{corollary}

\section{Transitive and subadditive functions}\label{s:subadditive}

Let $\theta,\l$ be regular infinite cardinals.
We say that a function $f:[\theta]^2\to \l$ is \emph{subadditive} if for all $\a<\b<\g<\theta$:
\begin{itemize}
	\item[(S1)] $f(\a,\g)\le\max\set{f(\a,\b),f(\b,\g)}$,
	\item[(S2)] $f(\a,\b)\le\max\set{f(\a,\g),f(\b,\g)}$.
\end{itemize}
If $f$ satisfies only condition (S1) we say it is \emph{transitive}.

We will now explain how to obtain a topological group from any transitive function.
So let $\l\le\k$ be regular infinite cardinals, and suppose we are given a transitive function $r:[\k^+]^2\to\l$.
For any $\xi<\l$ and $\b<\k^+$ let us define $F^r_{\xi}(\b)=\set{\a<\b:f(\a,\b)\le\xi}$ and denote $\mathcal F_r=\set{F^r_{\xi}(\b):\xi<\l,\b<\k^+}$.
Note that in this case $r=\rho_{\mathcal F_r}$.

\begin{theorem}\label{t:transitivegroup}
	The family $\mathcal F_r$ is a $(\k,\l)$-matrix.
\end{theorem}
\begin{proof}
	We have to prove that $\mathcal F_r$ satisfies conditions (G1)-(G4).
	
	For (G1), take $\a<\k^+$.
	First observe that $F^r_{\xi}(\a) \sst \a$ for each $\xi<\l$.
	For the other inclusion, take any $\g < \a$.
	Let $\eta=r(\g,\a)$.
	Then $\g \in F^r_{\eta}(\a)$, so we proved $\a=\bigcup_{\xi<\l}F^r_{\xi}(\a)$.
	
	For the proof of (G2), take any $\a < \k^+$ and any $\xi \leq \eta < \l$.
	If $\g \in F^r_{\xi}(\a)$, then $r(\g,\a) \leq \xi \leq \eta$.
	Thus, $\g \in F^r_{\eta}(\a)$.
	
	Next, we show (G3).
	Fix $\xi < \l$ and $\a \leq \b < \k^+$.
	We need to find $\eta < \l$ such that $F^r_{\xi}(\a) \sst F^r_{\eta}(\b)$.	
	Let $\eta = \max\{\xi,r(\a,\b)\}$.
	We will show this $\eta$ works.
	Let $\g \in F^r_{\xi}(\a)$.
	Then $r(\g,\a) \leq \xi$.
	Now, by (S1),
	$$
	r(\g,\b) \leq \max\{r(\g,\a),r(\a,\b)\} \leq \max\{\xi,r(\a,\b)\} = \eta,
	$$
	and therefore $\g \in F^r_{\eta}(\a)$.
\end{proof}

\begin{theorem}\label{t:subadditivegroup}
	If $r$ is moreover subadditive, then $\mathcal F_r$ is a strong $(\k,\l)$-matrix.
\end{theorem}
\begin{proof}
	By the previous theorem we still have to show (G4).
	To see that (G4) is true for $\mathcal F_r$, fix $\eta < \l$ and $\a \leq \b < \k^+$.
	We need to find $\xi < \l$ such that $F^r_{\eta}(\b) \cap \a \sst F^r_{\xi}(\a)$.
	Let $\xi = \max\{\eta,r(\a,\b)\}$.
	We will show that this $\xi$ works.
	Let $\g \in F^r_{\eta}(\b) \cap \a$.
	Then $r(\g,\b) \leq \eta$ and $\g < \a$. Now, by (S2),
	$$
	r(\g,\a) \leq \max\{r(\g,\b),r(\a,\b)\} \leq \max\{\eta, r(\a,\b)\} = \xi.
	$$
	Thus $\g \in F^r_{\xi}(\a)$, as required.
\end{proof}

By Theorem \ref{t:tgroup}, Theorem \ref{t:character}, Theorem \ref{t:smalltype}, and Theorem \ref{t:smallcharacter}, we have the following. 

\begin{corollary}\label{c:subadditive}
	The group $([\k^+]^{<\o}, \triangle)$ with the topology generated by $\cu(\mathcal F_r)$ becomes a topological group of character $\k^+$ whose only Tukey quotients of size less than $\k^+$ are $1$ and $\l$.
	Moreover, if $r$ is subadditive, then for $\d < \k^+$, the topology of the subgroup $[\d]^{<\o}$ is generated by the family $\cu_\d(\mathcal F_r)$.
\end{corollary}

We denote this group by $G_{r}(\k^+)=G(\mathcal F_r)$ (the former notation appeared in the earlier literature and the latter one is from this paper).

By Theorem \ref{t:frechet} we have another corollary:

\begin{corollary}
	If $r$ is subadditive, then it is $\k^+$-unbounded iff $G_r(\k^+)$ is $\l$-Frech\'et.
\end{corollary}

Note that, by Theorem \ref{t:maxtype}, we also have:

\begin{corollary}
	If $r$ satisfies condition (H), then $G_r(\k^+)\equiv_T [\k^+]^{<\o}$.
\end{corollary}

By Corollary~\ref{c:nonmetrizable} and Corollary~\ref{l:alpha_1}, the following holds:

\begin{corollary}
	For any subadditive $r:[\o_1]^2\to\o$, the group $G_r(\o_1)$ is a non-metrizable, $(\a_1)$-group.
\end{corollary}

There are various subadditive functions.
For the first example, consider the $C$-sequence $\seq{C_{\a}:\a<\o_1}$, and define the function $\rho:[\o_1]^2\to \o$ by $$\rho(\a,\b)=\max\set{\abs{C_{\b}\cap\a},\rho(\a,\min(C_{\b}\setminus\a)),\rho(\xi,\a):\xi\in C_{\b}\cap\a},$$ with the boundary condition $\rho(\a,\a)=0$.
This function has been introduced in \cite[Section 3]{walks}.
There it has been proved that $\rho$ satisfies (S1) and (S2), and that it even satisfies additional condition (H).
Another function introduced in \cite{walks} is $\rho_1:[\o_1]^2\to\rho$ defined by $\rho_1(\a,\b)=\max\set{\abs{C_{\b}\cap\a},\rho_1(\a,\min(C_{\b}\setminus\a))}$, with the boundary condition $\rho_1(\a,\a)=0$.
From the definition is clear that \begin{equation}\label{eq:1}\rho(\a,\b)\ge\rho_1(\a,\b)\mbox{ for }\a\le\b<\o_1.\end{equation}
Hence we know that $G_{\rho}(\o_1)$ has the following properties:
\begin{enumerate}
	\item $G_{\rho}(\o_1)$ is an $(\alpha_1)$-group by Corollary \ref{c:subadditive},
	\item $G_{\rho}(\o_1)$ is Frech\'et by Lemma 2.2.6 in \cite{walks} and Theorem \ref{t:frechet}, using (\ref{eq:1}),
	\item $G_{\rho}(\o_1)\equiv_T [\o_1]^{<\o}$ by Corollary \ref{c:maxtypeH}.
\end{enumerate}

Recall that the function $\bar\rho:[\o_1]^2\to\o$ introduced also in \cite[Section 3]{walks}, has the same properties, so the same holds for $G_{\bar\rho}(\o_1)$.
For the completeness of the paper, we give its definition: $\bar\rho(\a,\b)=2^{\rho(\a,\b)}(2\abs{\set{\xi\le\a:\rho(\xi,\a)\le\rho(\a,\b)}}+1)$.
The key feature of $\bar\rho$, distinguishing it from $\rho$, is that $\bar\rho(\a,\g)\neq\bar\rho(\b,\g)$ holds for all $\a<\b<\g<\o_1$.

Let us provide another example of transitive function, which then gives the corresponding group.
For the remainder of this section fix a regular infinite cardinal $\k$.
Let $\mathcal{TO}=\set{a_{\a}: \a < \kappa^+}$ be a $\k^+$-tower in $\mathcal{P}(\k)$, i.e. a family of sets $a_{\a} \sst \k$ ($\a<\k^+$) such that, for $\a,\b<\k^+$:

\begin{center}
	$\abs{a_{\a}\setminus a_{\b}}<\k$ if and only if $\a \le \b$.
\end{center}
Let us define a function $\rho_{\mathcal{TO}}:[\k^+]^2\to \k$ by $\rho_{\mathcal{TO}}(\a,\b)=\min\set{\xi<\k:a_{\a}\setminus \xi\sst a_{\b}}$.
Since $\mathcal{TO}$ is a $\k^+$-tower, the function $\rho_{\mathcal{TO}}$ is well defined.
As it is commented in \cite[pp. 64]{walks}, the function $\rho_{\mathcal{TO}}$ is transitive, so by Theorem \ref{t:tgroup}, Theorem \ref{t:character}, Theorem \ref{t:smalltype}, and Theorem \ref{t:transitivegroup}, we know that the topological group $G_{\rho_{\mathcal{TO}}}(\k^+)$ has character $\k^+$, and its only Tukey quotients of cardinality less then $\k^+$ are $1$ and $\k$.
Recall that a tower $\mathcal{TO}=\set{a_{\a}:\a<\o_1}$ in $\mathcal P(\o)$ is \emph{Hausdorff} if for every $n<\o$ and $\b<\o_1$, the set $F^{\mathcal{TO}}_n(\b)=\set{\a<\b:a_{\a}\setminus a_{\b}\sst n}$ is finite.
So, by Theorem \ref{t:maxtype} and Remark \ref{r:funcmatr}, if $\mathcal{TO}$ is a Hausdorff tower in $\mathcal{P}(\o)$, then $[\o_1]^{<\o}\equiv_T G(\mathcal F_{\mathcal TO})$.

\section{Trees}\label{s:tree}

In this section we will give two constructions of topological groups from trees.
The key difference is that the first construction is for trees of arbitrary height, while the other one is for trees of height $\o_1$.
The first one is more general, however we are able to obtain finer analysis of cofinal types of corresponding groups using the second construction.

Recall that a strict partial order $(T,\prec)$ is \emph{a tree} if it has a minimum called the root of $T$, and for each node $t\in T$ the set of its predecessors $t_{\downarrow}=\set{s\in T: s\prec t}$ is well-ordered.
For every $t\in T$, we define its height as $\h(t)=\otp(t_{\downarrow},\prec)$.
For an ordinal $\a$, the $\a$th level of $T$ is $T_{\a}=\set{t\in T:\h(t)=\a}$.
The height of the tree $T$ is $\h(T)=\min\set{\a:T_{\a}=\ps}$.
For a set of ordinals $\Gamma$ and any subset $X$ of $T$, we set $X\rest\Gamma=\bigcup_{\g\in \Gamma}X\cap T_{\g}$.
We say that $B\sst T$ is a \emph{branch} in $T$, if $B$ is linearly ordered and $T_{\a}\cap B\neq\ps$ for each $\a<\h(T)$.
If $\k$ is an infinite cardinal, than $T$ is a \emph{$\k$-tree} if it is of height $\k$ and each level of $T$ has cardinality less then $\k$, i.e. $\h(T)=\k$ and $\abs{T_{\a}}<\k$ for all $\a<\k$.
Every tree $T$ we consider in this paper will be \emph{normal}, i.e whenever $s\in T$ and $\h(s)<\a<\h(T)$ there is $t\in T_{\a}$ such that $s\prec t$.
An $\o_1$-tree $T$ is \emph{Aronszajn} if it has no brach, i.e. every linearly ordered subset is countable.
The following is a well-known result about Aronszajn trees which will be used.
Implicitly, it was first noticed by Jensen in \cite{devlin}, and the proof is a modification of the proof of Lemma 5.9 in \cite{handbook}.
We provide the proof for completeness of the paper.

\begin{lemma}\label{l:handbook}
	Let $T$ be an Aronszajn tree.
	Then for each uncountable pairwise disjoint family $X$ of finite subsets of $T$, the set $$\Gamma=\set{\g<\o_1:(\forall K\in [T_{\g}]^{<\o})\ \abs{\set{x\in X: (\forall t\in x)\ t_{\downarrow}\cap K=\ps}}=\aleph_1}$$ is uncountable. 
\end{lemma}
\begin{proof}
	We can assume that $X\sst [T]^n$ for some $n<\o$.
	Suppose the contrary, that $\Gamma$ is countable.
	Then for each $\a>\sup(\Gamma)$ there is a finite set $K_{\a}\sst T_{\a}$ so that $X_{\a}=\set{x\in X:(\forall t\in x)(\forall s\in K_{\a})\ s\not\prec t}$ is countable.
	This means that there is an uncountable set $\Sigma\sst \o_1\setminus\sup(\Gamma)$ and $m<\o$ so that $\abs{K_{\a}}=m$ for each $\a\in\Sigma$.
	Let $\cu$ be a uniform ultrafilter on $\Sigma$, i.e. $\cu$ is an ultrafilter on $\Sigma$ such that each element of $\cu$ is uncountable.
	Enumerate $X=\set{x_{\a}:\a\in\Sigma}$ and adopt a convention in this proof that for a finite set $L\sst\o_1$ and $i<\o$, we denote the $i$th element of $L$ in its increasing enumeration by $L(i)$.
	For $\d\in\Sigma$ and $(i,j)\in m\times n$, define $A_{\d}(i,j)=\set{\a\in\Sigma:K_{\d}(i)\prec x_{\a}(j)}$.
	Then for each $\d\in\Sigma$, since $X_{\d}$ is countable and $\cu$ is uniform, the set $A_{\d}=\set{\a\in\Sigma:x_{\a}\notin X_{\d}}$ is in $\cu$.
	As $A_{\d}=\bigcup_{(i,j)\in m\times n}A_{\d}(i,j)$, we can find some $i_{\d}$ and $j_{\d}$ such that $A_{\d}(i_{\d},j_{\d})\in\cu$.
	Then for some $(i,j)\in m\times n$ we have $i_{\d}=i$ and $j_{\d}=j$ for $\d$ in an uncountable $\Delta\sst\Sigma$.
	Consider the set $B=\set{K_{\d}(i):\d\in\Delta}$.
	Take $\b<\d$ in $\Delta$ and let $\a\in A_{\b}(i,j)\cap A_{\d}(i,j)\in \cu$ be above both $\a$ and $\b$ (this is possible as $\cu$ is uniform).
	Then both $K_{\b}(i)\prec x_{\a}(j)$ and $K_{\d}(i)\prec x_{\a}(j)$.
	Since $T$ is a tree and $\b<\d$ it must be that that $K_{\b}(i)\prec K_{\d}(i)$, i.e. $(B,\prec)$ is an uncountable linearly ordered subset of $(T,\prec)$ which is in contradiction with the assumption that $T$ is Aronszajn.
\end{proof}

\subsection{The general construction}

In this subsection we assume that $\k$ is a regular infinite cardinal and that $T$ is a $\k^+$-tree.
Suppose that $T$ is \emph{$\k$-splitting}, i.e. every node in $T$ admits at least $\k$ many immediate successors, which becomes exactly $\k$ by the assumption that $T$ is a $\k^+$-tree.
So we will assume that $T=(\k^+,\prec)$ and that $T_{\a}=[\k\a,\k\a+\k)$ for $\a<\k^+$.
Then the set $C=\set{\a<\k^+:\k\a=\a}$ contains a club in $\k^+$.
In particular, it is stationary in $\k^+$ (which again implies it is cofinal).
Now, for each $\xi < \k$ and $\a\in C$ we define $$F^T_{\xi}(\a)=\set{\g<\a:(\exists \eta\le\xi)\ \g\prec \a+\eta},$$ i.e. if $\k\a=\a$ and $\xi<\k$, then $F^T_{\xi}(\a)$ consists of all the predecessors of all the nodes on the $\a$th level which appear no later than $\a+\xi$ in the enumeration of that level.
Finally, let us denote $\mathcal F_T=\set{F^T_{\xi}(\a):\xi<\k,\a\in C}$.

\begin{lemma}\label{l:treematrix}
	The family $\mathcal F_T$ is a strong $\k$-matrix indexed by $C$.
\end{lemma}
\begin{proof}
	As we already mentioned, $C$ is cofinal in $\k^+$, so we have to show that properties (G1)-(G4) hold for $\mathcal F_T$.

	We start with (G1).
	First note that by definition $F^T_{\xi}(\a)\sst\a$ for each $\xi<\k$ and $\a\in C$.
	To prove the needed equality, take any $\g<\a=\k\a$.
	Then $\h(\g)<\a$, so by normality of $T$, there is some $\eta'<\k$ such that $\g\prec\a+\eta'$.
	So $\g\in F^T_{\eta'}(\a)$ as required.

	Now we show (G2).
	Take any $\a\in C$ and $\xi\le\eta<\k$.
	Take $\g\in F^T_{\xi}(\a)$.
	This means that $\g\prec\a+\xi'$ for some $\xi'\le\xi$.
	Since $\xi'\le\xi\le\eta$, this implies $\g\in F^T_{\eta}(\a)$, as required.

	Next, we prove (G3).
	Fix $\xi<\k$ and $\a\le\b$ in $C$.
	Recall that this means that, in particular, $\k\a=\a$ and $\k\b=\b$.
	For each $\varepsilon \le \xi$ there is $\eta_{\varepsilon} < \k$ such that $\a+\varepsilon\prec\b+\eta_{\varepsilon}$.
	Let $\eta = \sup\{\eta_{\varepsilon} : \varepsilon \le \xi\} < \k$.
	Then we have $F^T_{\xi}(\a)\sst F^T_{\eta}(\b)$.

	Finally, we prove (G4).
	Fix $\a\le\b$ in $C$ and $\eta<\k$.
	We need an ordinal $\xi<\k$ so that $F^T_{\eta}(\b)\cap \a\sst F^T_{\xi}(\a)$.
	First note that for each $\varepsilon \le \eta$, since $T$ is a tree, there is a unique $\xi_{\varepsilon} < \k$ such that $\a+\xi_{\varepsilon} \prec \b+\varepsilon$.
	Let $\xi = \sup\{\xi_{\varepsilon} : \varepsilon \le \eta\} < \k$.
	We will show that $\xi$ is as required.
	Let $\g\in F^T_{\eta}(\b)\cap \a$.
	This means that $\g\prec\b+\varepsilon_0$ for some $\varepsilon_0\le\eta$.
	Denote $\d=\h(\g)$.
	Since $\k\a=\a$, we know that $\d<\a$.
	Consider the node $\a+\xi_{\varepsilon_0}$.
	We know that $\a+\xi_{\varepsilon_0}\prec \b+\varepsilon_0$ and $\g\prec \b+\varepsilon_0$.
	Since $\g<\a$ and $T$ is a tree, in particular $((\b+\varepsilon_0)_{\downarrow},\prec)$ is linearly ordered, it must be that $\g\prec\a+\xi_{\varepsilon_0}$.
	As $\xi_{\varepsilon_0}\le\xi$ we proved $\g\in F^T_{\xi}(\a)$. 
	So $F^T_{\eta}(\b)\cap\a\sst F^T_{\xi}(\a)$.
\end{proof}

Now, by Theorem \ref{t:tgroup}, Theorem \ref{t:character}, Theorem \ref{t:smalltype}, and Theorem \ref{t:smallcharacter}, we have the following corollary. 

\begin{corollary}\label{c:treecor}
	For any $\k^+$-tree $T$, the group $([\k^+]^{<\o},\triangle)$ with the topology generated by $\cu(\mathcal F_T)$ becomes a topological group of character $\k^+$ whose only Tukey quotients of size less than $\k^+$ are $1$ and $\k$.
	Moreover, for each $\d<\k^+$, the topology of the subgroup $[\d]^{<\o}$ is generated by the family $\cu_{\d}\left(\mathcal F_T\right)$.
\end{corollary}

Note the following corollary which follows from Corollary \ref{c:nonmetrizable} and Corollary \ref{c:alpha_1}.

\begin{corollary}\label{c:metrizable}
    Let $T$ be an $\o_1$-tree. Then $G(\mathcal F_T)$ is a non-metrizable, $(\alpha_1)$-group.
\end{corollary}

The next theorem shows that the notion of an Aronszajn tree can be described using our topological notions.

\begin{theorem}\label{t:treefrechet}
    If $T$ is an $\o_1$-tree, then the group $G(\mathcal F_T)$ is Frech\'et if and only if the tree $T$ is Aronszajn.
\end{theorem}

\begin{proof}
	Before we start the proof, recall that $\mathcal F=\set{F^T_n(\a)
	:\a\in C,\ n<\o}$ where $C=\set{\a<\o_1:\o\a=\a}$.

	To prove the theorem, first assume that $G(\mathcal F_T)$ is a Frech\'et group, and suppose the contrary, that $T$ is not an Aronajn tree, i.e. there is a branch $B\sst T$.
    Let $A=\set{\set{\g}:\g\in B}$.
    First we show that $\ps \in \overline{A}$.
    Fix an open neighborhood $U^T_{n}(\a)$ of $\ps$.
	For any $\b>\a$ in $C$ there is $m<\o$ such that $\b+m\in B$.
	Then we have $\b+m\notin F^T_{n}(\a)$, implying $\set{\b+m}\in U^T_{n}(\a)\cap A$ as required.
	Since $G(\mathcal F_T)$ is Frech\'et, there is a sequence $\seq{\set{\g_n}:n<\o}$ of elements in $A$ such that $\lim\set{\g_n}=\ps$.
	Since $B$ is a branch, and $T$ is an $\o_1$-tree, there are $k<\o$ and $\d$ in $C$ such that $\g_n\prec\d+k$ for each $n<\o$.
	Then $\g_n\in F^T_{k}(\d)$ for each $n<\o$, i.e. $\set{\g_n}\notin U^T_{k}(\d)$ for all $n<\o$ contradicting the assumption that $\seq{\set{\g_n}:n<\o}$ is converging to $\ps$.

	Now suppose that $T$ is an Aronszajn tree.
	We prove that $G(\mathcal F_T)$ is then Frech\'et.
	Take any $A\sst G(\mathcal F_T)$ such that $\ps\in\overline{A}$.
	Let $\theta>\left(2^{\o_1}\right)^+$ be a regular cardinal and let $M$ be a countable elementary submodel of $H\left(\theta\right)$ such that $T,\mathcal F_T,C,A\in M$ and $M\cap \o_1\in C$.
	This is possible by Lemma \ref{l:elsubmodel}.
	Denote $\d=M\cap \o_1$.
	As in the proof of Theorem \ref{t:frechet}, it is enough to prove that $\ps\in\overline{A\cap M}$.
	If $A$ is countable, then $A\sst M$ and we are done.
	So assume $A$ is uncountable, and also assume that for some $k<\o$ and all $a\in A$ we have $\abs{a}=k$.
	By Theorem \ref{t:smallcharacter} (see also Corollary \ref{c:treecor}), it is enough to prove that for each $n<\o$ there is some $a\in A\cap M\cap U^T_n(\d)$.
	So fix $n<\o$.
	Clearly, there is $y\in A\cap U^T_n(\d)$.
	If $y\in M$, the proof is finished.
	So assume that $y\notin M$.
	Let $y_0=y\cap M$.
	Denote $\g_0=\max\set{\h(t):t\in y_0}$
	Then $y_0$ is an element of $M$ as a finite subset of $M$.
	In the same way as in the proof of Theorem \ref{t:frechet} one can show that for each $\a<\o_1$ there is $w_{\a}$ such that $\a\cap w_{\a}=\ps$ and $w_{\a}\cup y_0\in A$.
	Moreover, as in that proof we can take $A_{y_0}=\set{w_{\a}:\a<\o_1}$ to be an element of $M$. 
	
	Since, $y\setminus y_0\in A_{y_0}$ and $y\setminus y_0\notin M$, the family $A_{y_0}$ is uncountable.
	By the choice of $w_{\a}$ for $\a<\o_1$, the family $A_{y_0}$ is an uncountable family of pairwise disjoint finite subsets of $\o_1$.
	Consider the set $\Gamma$ of ordinals $\g$ such that $\abs{\set{a\in A_{y_0}: (\forall t\in a)\ t_{\downarrow}\cap E=\ps}}=\aleph_1$ for each finite $E\sst T_{\g}$.
	By Lemma \ref{l:handbook}, this set is cofinal in $\o_1$.
	Since $\Gamma$ is defined only using parameters in $M$, by elementarity we have $\Gamma\in M$.
	Thus, by elementarity, $\Gamma\cap M$ is cofinal in $\d$.
	Take $\g'\in (\Gamma\cap M)\setminus (\g_0+1)$.
	Let $E=F^T_n(\d)\cap T_{\g'}$.
	Then $E$ is a finite subset of $T_{\g'}\sst M$, so $E\in M$.
	Since $\g'\in\Gamma$, the family $$A^{\g'}_{y_0}=\set{a\in A_{y_0}: \g'<a\ \wedge\ (\forall t\in a)\ t_{\downarrow}\cap E=\ps}$$ is uncountable.
	Since $A^{\g'}_{y_0}$ is defined only using parameters in $M$, it is also an element of $M$.
	Since $A^{\g'}_{y_0}$ is uncountable, it is in particular non-empty, so by elementarity of $M$, there is $z\in A^{\g'}_{y_0}\cap M$. 
	Then $z\cap F^T_n(\d)=\ps$.
	To see this suppose that $t\in F^T_n(\d)$ for some $t\in z$.
	In this case $t\prec \d+m$ for some $m\le n$.
	This means that $t_{\downarrow}\sst F^T_n(\d)$.
	Since $\g'<z$ and $((\d+m)_{\downarrow},\prec)$ is well-ordered, then it must be that $t_{\downarrow}\cap E\neq\ps$ contradicting the choice of $z$.
	Hence, $y_0\cup z\in M\cap U_n(\d)$, and the group is Frech\'et.
\end{proof}

\subsection{The construction from a specific class of trees}

In this subsection we construct groups from a specific class of trees.
First, for a function $f\in\o^{<\o_1}$ we define its norm by $$\norm{f}=\sup\set{f(\g):\g<\dom(f)}.$$
For two functions $f,g\in\o^{<\o_1}$ we define their difference $f-g\in\o^{<\o_1}$ as follows: $\dom(f-g)=\dom(f)\cap \dom(g)$ and $(f-g)(\g)=\abs{f(\g)-g(\g)}$.
Finally, we say that an $\o_1$-tree $T\sst \o^{<\o_1}$ is \emph{$\ell_{\infty}$-tree} if $\norm{s-t}<\o$ for all $s$ and $t$ in $T$.
Recall that a tree $T\sst \o^{<\o_1}$ is \emph{coherent} if the set $\set{\g\in\dom(s-t):s(\g)\neq t(\g)}$ is finite for all $s,t\in T$.
Then every coherent $\o_1$-tree $T\sst \o^{<\o_1}$ is an $\ell_{\infty}$-tree.
For another example, consider $T(\rho_2)=\set{\rho_2(\cdot,\b)\rest\a:\a\le\b<\o_1}$.
Here, for a $C$-sequence $\seq{C_{\a}:\a<\o_1}$, we define a function $\rho_2:[\o_1]^2\to\o$ by $\rho_2(\a,\b)=\rho_2\left(\a,\min\left(C_{\b}\setminus\a)\right)\right)+1$ and boundary condition $\rho_2(\a,\a)=0$.
Then $T(\rho_2)$ is a subtree of $\o^{<\o_1}$, and it is an $\ell_{\infty}$-tree by \cite[Lemma 2.4.2]{walks}.

For the remainder of this subsection, we fix an $\o$-splitting $\ell_{\infty}$-tree $T$.
For $t\in T$ and $n<\o$ we define $$F^{\infty}_n(t)=\set{s\in T\rest \h(t):\norm{s-t}\le n}.$$
Let us denote $\ft=\set{F^{\infty}_n(t):n<\o,t\in T}$.
Clearly, $\ft$ is not an $\o$-matrix according to our definition, as our matrices are indexed by sets of ordinals.
However, it can be regarded as a strong $(T,\o)$-matrix in a sense that it satisfies:
\begin{enumerate}
	\item $\bigcup_{n<\o}F^{\infty}_n(t)=T\rest\h(t)$ for each $t\in T$,
	\item $F^{\infty}_m(t)\sst F^{\infty}_n(t)$ for each $t\in T$ and $m\le n<\o$,
	\item for all $\a\le\b<\o_1$, $s\in T_{\a}$, $t\in T_{\b}$, and $m<\o$, there is $n<\o$ so that $F^{\infty}_m(s)\sst F^{\infty}_n(t)$,
	\item for all $\a\le\b<\o_1$, $s\in T_{\b}$, and $m<\o$, there are $t\in T_{\a}$ and $n<\o$ such that $F^{\infty}_m(s)\rest\d\sst F^{\infty}_n(t)$.
\end{enumerate}
Theorem \ref{t:treealt-group} shows it is essentially a matrix (i.e. the proof of (1-3) is contained in the proof of that lemma), while Theorem \ref{t:treealt-strong} explains why one can think of it as a strong matrix (i.e. its proof contains the proof of (4)).
As before, for $n<\o$ and $t\in T$ define $U^{\infty}_n(t)=\set{a\in [T]^{<\o}:a\cap F^{\infty}_n(t)=\ps}$ and denote $$\cu^{\infty}_T=\set{U^{\infty}_n(t):n<\o,t\in T}.$$

\begin{theorem}\label{t:treealt-group}
	The group $([T]^{<\o},\triangle)$ when equipped with $\cu^{\infty}_T$ as the local base of the identity, becomes a topological group.
\end{theorem}
\begin{proof}
	We check properties (1) to (6) of Theorem \ref{t:osnovna}.
	Conditions (1), (2), (3), and (4) are true in the same way as in the proof of Theorem \ref{t:tgroup}.
	
	Let us prove (5).
	Take $U^{\infty}_n(t)$ and $U^{\infty}_m(s)$ in $\cu^{\infty}_T$.
	Let $\a$ and $\b$ in $\o_1$ be such that $s\in \o^{\a}$ and $t\in \o^{\b}$.
	Let $k=\norm{t-s}+\max\set{m,n}$ and suppose without any loss of generality that $\a\le\b$.
	It is clear that $F^{\infty}_n(t)\sst F^{\infty}_k(t)$.
	So it is enough to show that $F^{\infty}_m(s)\sst F^{\infty}_k(t)$.
	Take any $u\in F^{\infty}_m(s)$.
	This means that $\norm{s-u}\le m$, i.e. $(\forall g\in\dom(s-u))\ \abs{s(\g)-u(\g)}\le m$.
	From $\a\le\b$ we know $\dom(s-u)=\dom(t-s)$.
	Take any $\g\in \dom(s-u)$.
	Then $$\abs{t(\g)-u(\g)}\le \abs{t(\g)-s(\g)}+\abs{s(\g)-u(\g)}\le \norm{t-s}+m\le k.$$
	This implies that $\g\in F^{\infty}_k(t)$, as required.
	
	We still have to show (6).
	Suppose on the contrary, that $a\in\bigcap\cu^{\infty}_T$ for some $a\in [T]^{<\o}\setminus\set{\ps}$.
	Since $a\neq\ps$, there is some $s\in a$. 
	By normality of $T$, there is $t\in T$ such that $s\subsetneq t$, i.e. $s\in t_{\downarrow}$.
	Then $\norm{t-s}=0$, so $s\in a\cap F^{\infty}_0(t)$.
	This means that $a\notin U^{\infty}_0(t)$, contradicting the assumption.
	So (6) is true as well.
\end{proof}

We will denote the topological group from the previous theorem by $G(\ft)$.

\begin{theorem}
	The group $G(\ft)$ has character $\o_1$.
\end{theorem}
\begin{proof}
	The proof is almost identical to the proof of Theorem \ref{t:character}.
	Suppose on the contrary, that $\mathcal B=\set{U^{\infty}_{n_k}(t_k):k<\o}\sst \cu^{\infty}_T$ is a local base of the identity in $G(\ft)$.
	Since $B=\set{\h(t_k):k<\o}\sst\o_1$ is countable, there is $s\in T$ such that $\h(s)>\sup(B)$.
	By normality there is $t\in T$ such that $s\subsetneq t$.
	Then, as before, $\set{s}\in U^{\infty}_n(t_k)$ for all $n,k<\o$ and $s\in F^{\infty}_0(t)$.
	This means that $\set{s}\notin U^{\infty}_0(t)$, hence no $U^{\infty}_{n_k}(t_k)$ can be a subset of the basic open $U^{\infty}_0(t)$, contradicting the assumption that $\mathcal B$ is a local base of $\ps$.
\end{proof}

As we have already mentioned, the next lemma shows why $\ft$ can be regarded as a strong matrix.
Note a difference that when considering the restriction of the group to some $\d<\o_1$, here the local base of the restriction is two-dimensional in a sense that it depends on all countably many elements from $\d$th level and all $n<\o$.
Whereas in the case of a strong matrix, the local base of the restriction depends on only one element in $\o_1$ and all $n<\o$.

\begin{theorem}\label{t:treealt-strong}
	For $\d<\o_1$, the set $\cu^{\infty}_{T}(\d)=\set{[T\rest\d]^{<\o}\cap U^{\infty}_n(t):n<\o,t\in T_{\d}}$ is a local base of $\ps$ in the topological group $([T\rest\d]^{<\o},\triangle)$ with the topology induced from $G(\ft)$.
\end{theorem}
\begin{proof}
	The proof of this theorem follows closely the proof of Theorem \ref{t:smallcharacter}.
	Fix an ordinal $\d<\o_1$.
    It is enough to show that for each $m<\o$ and $s\in T$ there are $n<\o$ and $t\in T_{\d}$ so that $[T\rest\d]^{<\o}\cap U^{\infty}_n(t)\sst [T\rest\d]^{<\o}\cap U^{\infty}_m(s).$
    To show this, it is enough to prove that $F^{\infty}_m(s)\rest\d\sst F^{\infty}_n(t)$ for suitably chosen $n$ and $t$.
    Let us consider two cases.

    First, that $\h(s)\le\d$.
	Then, by normality of $T$, there is $t\in T_{\d}$ such that $s\sst t$.
	We will prove that $F^{\infty}_m(s)\sst F^{\infty}_m(t)$.
	Take any element $u$ in $F^{\infty}_m(s)$.
	Then we have $\dom(u)=\dom(s-u)=\dom(t-u)$.
	Since $s\sst t$ we also know that $(\forall \g\in\dom(u))\ \abs{t(\g)-u(\g)}=\abs{s(\g)-u(\g)}\le m$.
	This implies that $u\in F^{\infty}_m(t)$.

	Second case is that $\d<\h(s)$.
	Then take $t=s\rest\d$.
	Again we will prove that $F^{\infty}_m(s)\rest\d\sst F^{\infty}_n(t)$.
	Take $u\in F^{\infty}_m(s)\rest\d$.
	Then $\dom(u)=\dom(t-u)=\dom(s-u)$ and $(\forall \g\in\dom(u))\ \abs{t(\g)-u(\g)}=\abs{s(\g)-u(\g)}\le m$ (this time because $t\sst s$).
	Hence, again $u\in F^{\infty}_m(t)$.
\end{proof}

\begin{corollary}\label{c:treealtnonm}
	The group $G(\ft)$ is a non-metrizable group which is an increasing union of metrizable groups. In particular, $G(\ft)$ is an $(\a_1)$-group.
\end{corollary}

The description of tightness is not as nice as in the case of the definition of the topology from the previous subsection, however, we are able to prove some results.

\begin{lemma}
	If $G(\ft)$ is countably tight, then $T$ is Aronszajn.
\end{lemma}
\begin{proof}
	The same proof as in Theorem \ref{t:treefrechet} gives this lemma.
	To see this, suppose $T$ is not Aronszajn, and take a branch $B$.
	Consider $\mathcal A=\set{\set{s}:s\in B}$.
	Since $B$ is a branch $\set{\h(s):s\in B}$ is cofinal in $\o_1$.
	Thus $\ps\in\overline{\mathcal A}$.
	By countable tightness there is $A\in [B]^{<\o_1}$ such that $\ps\in\overline{\set{\set{s}:s\in A}}$.
	Pick $t\in B$ so that $\h(s)<\h(t)$ for each $s\in A$.
	Then $s\in F^{\infty}_0(t)$ for all $s\in A$ contradicting the choice of $A$.
\end{proof}

\begin{theorem}
	If $T$ is a Souslin $\ell_{\infty}$-tree, then $G(\ft)$ is a Frech\'et group.
\end{theorem}
\begin{proof}
	We use the same structure of the proof as in Theorem \ref{t:frechet} and Theorem \ref{t:treefrechet}.
	So suppose that $\mathcal A\sst G(\ft)$ is such that the identity $\ps$ belongs to the closure of $\mathcal A$.
	As before, take large enough $\theta$ and $M$ a countable elementary submodel of $H(\theta)$ containing $T$ and $\mathcal A$ as elements.
	Let $\d=M\cap\o_1$.
	It is enough to prove that $\ps$ belongs to the closure of $\overline{M\cap \mathcal A}$.
	By Theorem \ref{t:treealt-strong}, it is enough to show that for each $n<\o$ and $t\in T_{\d}$, there is $a\in M\cap\mathcal A$ such that $a\cap F^{\infty}_n(t)=\ps$.
	So we fix $n<\o$ and $t\in T_{\d}$, and work towards finding such an $a$.

	For the purpose of this proof, for a finite set $c\sst T$ denote $$\supp(c)=\set{\h(v):v\in c}.$$

	Since $\ps\in\overline{\mathcal A}$ there is $b\in \mathcal A$ such that $b\cap F^{\infty}_n(t)=\ps$.
	Denote $b_0=b\cap M$ and $b_1=b\setminus M$.
	Again, in the same way as in the proof of Theorem \ref{t:frechet} and Theorem \ref{t:treefrechet}, there is a function $f:\o_1\to [\o_1]^{<\o}$ in $M$ such that, for $\a<\o_1$, $f(\a)$ is minimal in the right-lexicographical ordering of $[\o_1]^{<\o}$ with the property that $\a<\supp(f(\a))$ and $b_0\cup f(\a)\in\mathcal A$.
	For $\b<\o_1$ consider the following subset of $T$:
	$$D_{\b}=\set{s\in T:(\exists \a_s>\b)\ \left[\h(s)>\supp(f(\a_s))\ \wedge\ (\forall v\in f(\a_s))\norm{v-s}>n\right]}.$$
	Clearly $D_{\b}$ is an open subset of $T$, meaning that for each $s\in D_{\b}$ and $u\supseteq s$ we know that $u\in D_{\b}$ as well.
	Let us prove that $D_{\b}$ is dense in $T$, meaning that for each $u\in T$ there is some $s\in D_{\b}$ such that $u\sst s$.
	So take $u\in T$.
	Let $\a$ be a minimal ordinal greater than both $\h(u)+1$ and $\b$.
	Then we can arrange an increasing enumeration $\set{\xi_i:i<k}$ of $\supp(f(\a))$.
	Since $\supp(f(\a))>\a$ we have $$\a<\xi_0<\xi_1<\cdots<\xi_{k-1}.$$
	We inductively define $s_i$ for $i<k$ so that:
	\begin{enumerate}
		\item $u\sst s_0$
		\item $\h(s_i)=\xi_i$ for $i<k$,
		\item $s_i\sst s_{i+1}$ for $i<k-1$,
		\item $\norm{s_i-v}>n$ for each $v\in f(\a)\cap T_{\xi_i}$ and $i<k$.
	\end{enumerate}
	Let us show that this is possible.
	First we explain it is possible to obtain $s_0$.
	Since $\a>\h(u)+1$ and $T$ is $\o$-splitting, there is some $$m_0>\max\set{v(\h(u)):v\in f(\a)\cap T_{\xi_0}}+n+1$$ such that $u^\frown\set{\seq{\h(u),m_0}}\in T$.
	By normality of $T$ and since $\h(u)+1<\a<\xi_0$, there is $s_0\in T_{\xi_0}$ such that $u^{\frown}\set{\seq{\h(u),m_0}}\sst s_0$.
	Then (1) and (2) are clear, and (3) is vacuous for $s_0$.
	The condition (4) holds for $s_0$ because from the choice of $m_0$ we know that $\abs{s_0(\h(u))-v(\h(u))}>n$ for all $v\in f(\a)\cap T_{\xi_0}$.
	Suppose now, for $i<k-1$, that $s_0,\dots,s_i$ have been constructed so that they satisfy (1-4).
	We need to define $s_{i+1}$.
	Again, since $\xi_i<\xi_{i+1}$ and $T$ is $\o$-splitting there is $$m_i>\max\set{v(\xi_i):v\in f(\a)\cap T_{\xi_{i+1}}}+n+1$$ such that $s_i^{\frown}\set{\seq{\xi_i,m_i}}\in T$.
	From normality of $T$ and $\xi_i+1\le\xi_{i+1}$ we conclude that there is some $s_{i+1}\in T_{\xi_{i+1}}$ such that $s_i^{\frown}\set{\seq{\xi_i,m_i}}\sst s_{i+1}$.
	Condition (1) is not relevant for $s_{i+1}$, while conditions (2) and (3) are clearly satisfied. 
	Same as before, (4) holds for $s_{i+1}$ by the choice of $m_i$ and definition of $\norm{\ }$.
	This concludes the proof that the sequence $\seq{s_0,\dots,s_{k-1}}$ is well-defined.

	Finally, by normality of $T$ take $s\in T_{\xi_{k-1}+1}$ such that $s_{k-1}\sst s$.
	By (2), then $s_i\sst s$ for each $i<k$.
	We claim that $s$ is as required.
	Clearly $u\sst s$ from (1) and (2).
	It is also evident that $\a>\b$ and $\h(s)>\supp(f(\a))$. 
	So we still have to show that $\norm{v-s}>n$ for each $v\in f(\a)$.
	To see this, take $v\in f(\a)$.
	Then $\h(v)=\xi_j$ for some $j<k$.
	Then, by (4) we have $\norm{s_j-v}>n$.
	Since $s_j\sst s$, this implies that $\norm{s-v}>n$, as required.
	Hence $D_{\b}$ is dense in $T$.
	For $s\in D_{\b}$ denote $\a_s$ the corresponding ordinal $\a_s>\b$ such that $\h(s)>\supp(f(\a_s))$ and that $\norm{v-s}>n$ for all $v\in f(\a_s)$.

	Since $D_{\b}$ is dense open, for $\b<\o_1$, and $T$ is a Souslin tree, there is a level $\a<\o_1$ such that $D_{\b}\rest (\o_1\setminus\a)=T\rest (\o_1\setminus\a)$.
	So for $\b=\max(\supp(b_0))\in M$, since in that case $D_{\b}\in M$, using elementarity of $M$ we get $s\in t_{\downarrow}\cap D_{\b}\cap M$.
	The corresponding $\a_s$ then belongs to $M$, and gives $f(\a_s)\in M$ such that $b_0\cup f(\a_s)\in \mathcal A\cap M$ and $(b_0\cup f(\a_s))\cap F^{\infty}_n(t)=\ps$.
	To see the last statement, suppose the contrary, that $v\in f(\a_s)\cap F^{\infty}_n(t)$ (note that $v\in b_0\cap F^{\infty}_n(t)$ is not possible by the choice of $b$).
	Then, since $s\sst t$ and $\norm{s-v}>n$ we get $\norm{t-v}>n$ contradicting the assumption that $v\in F^{\infty}_n(t)$.
	Hence, the theorem is proved.
\end{proof}

Finally, we can show that this group does not necessarily have maximal possible cofinal type.

\begin{theorem}\label{t:newtype}
	If $T$ is a Souslin $\ell_{\infty}$-tree, then $[\o_1]^{<\o}\not\le_T G(\ft)$.
\end{theorem}
\begin{proof}
	First observe that for $s\sst t$ in $T$, and $n<\o$, we have $F^{\infty}_n(s)\sst F^{\infty}_n(t)$.
	To see this, take $u\in F^{\infty}_n(s)$.
	Then $\dom(u)=\dom(s-u)=\dom(t-u)$, and from $s\sst t$ we conclude $(\forall \g\in\dom(u))\ \abs{t(\g)-u(\g)}=\abs{s(\g)-u(\g)}\le m$, as required.

	Now we prove the theorem.
	Suppose the contrary, that $f:[\o_1]^{<\o}\to \ft$ is a Tukey function.
	This is fine as, by Remark \ref{r:tukey}, $\ft\equiv_T \cu^{\infty}_T$.
	By the assumption that $f$ is Tukey, we know that $f\left[[\o_1]^{<\o}\right]$ is uncountable, otherwise there would be an uncountable set, hence unbounded, mapping to a point, hence unbounded set with bounded image which is impossible.
	If there were a bounded infinite subset of $f\left[[\o_1]^{<\o}\right]$, its inverse image would be infinite, hence unbounded in $[\o_1]^{<\o}$, and this would give a contradiction.
	Hence, to finish the proof of the theorem, it is enough to show that every uncountable subset of $\ft$ is bounded with respect to $\sst$.
	
	Let $\mathcal X\sst\ft$ be uncountable.
	Enumerate $T=\set{t_{\a}:\a<\o_1}$ in a way that if $t_{\a}\sst t_{\b}$, then $\a\le\b$.
	By the cardinality argument, there is $n<\o$ and an uncountable set $X\sst \o_1$ such that $\set{F^{\infty}_n(t_{\a}):\a\in X}\sst \mathcal X$.
	Note that $X$ has order type $\o_1$ and let us now define a coloring $[X]^2=K_0\cup K_1$ as follows: for $\a<\b$ in $X$ we set $\set{\a,\b}\in K_0$ iff $t_{\a}\not\sst t_{\b}$.
	Since $\o_1\rightarrow (\o_1,\o+1)^2$ by Theorem 11.3 in \cite[pp. 72]{erdos} and $\otp(X)=\o_1$, one of the following two options holds:
	\begin{enumerate}
		\item there is $\Gamma\sst X$ such that $\otp(\Gamma)=\o_1$ and $\set{\a,\b}\in K_0$ for $\a<\b$ in $\Gamma$,
		\item there is $\Delta\sst X$ so that $\otp(\Delta)=\o+1$ and $\set{\a,\b}\in K_1$ for $\a<\b$ in $\Delta$.
	\end{enumerate}
	Case (1) does not happen, as it would imply that $t_{\a}\not\sst t_{\b}$ and $t_{\b}\not\sst t_{\a}$ for all $\a<\b$ in an uncountable $\Gamma$.
	This would then mean that there is an uncountable antichain in $T$, which is not possible by the assumption that $T$ is a Souslin tree.
	So case (2) is true.
	This means that we have a set $\set{t_{\a_i}:i\le\o}\sst T$ such that $\a_i<\a_j$ and $\set{\a_i,\a_j}\in K_1$ for $i<j\le\o$, i.e. $t_{\a_i}\sst t_{a_j}$ or $t_{\a_j}\sst t_{\a_i}$ for $i<j\le\o$.
	By the assumption on the enumeration, $t_{\a_j}\sst t_{\a_i}$ for $i<j$ is not possible as then $\a_j<\a_i$.
	So we have that $t_{\a_i}\sst t_{\a_j}$ for $i<j\le\o$.
	In particular $t_{\a_i}\sst t_{\a_{\o}}$ for each $i<\o$.
	By the observation at the beginning of the proof this implies that $F^{\infty}_n(t_{\a_i})\sst F^{\infty}_n(t_{\a_{\o}})$ for $i<\o$.
	Hence there is an infinite bounded set $\set{F^{\infty}_n(t_i):i<\o}$ in $\mathcal X$. 
\end{proof}

\begin{remark}
	Theorem \ref{t:newtype} not only gives the consistent counterexample to the mentioned problem of Feng, but also provides a new example of a cofinal type of cofinality at most $\o_1$ which is not equivalent to any of the five: $1$, $\o$, $\o_1$, $\o\times\o_1$, or $[\o_1]^{<\o}$.
	To see that $G(\mathcal F^{\infty}_T)\not\equiv_T \o\times\o_1$, observe that as the group in question is Frech\'et, by \cite[Corollary 4.11]{km} it would have to be metrizable, which is not true by Corollary \ref{c:treealtnonm}.
	Note that this example is obtained from a Souslin $\ell_{\infty}$-tree.
\end{remark}

\section{Gaps}\label{s:gaps}

In this section, let $\k$ be a regular infinite cardinal.
Let us also fix a $(\k^+,\k^+)$ pre-gap $\mathcal G=\set{(a_{\a},b_{\a}):\a<\k^+}$ in $\mathcal P(\k)$, i.e. a family of sets $a_{\a},b_{\a}\sst\k$ such that:
\begin{enumerate}
	\item $\abs{a_{\a}\setminus a_{\b}}<\k$ and $\abs{b_{\a}\setminus b_{\b}}<\k$ for each $\a<\b<\k^+$,
	\item $a_{\a}\cap b_{\a}=\ps$ for each $\a<\k^+$,
\end{enumerate}

Recall that a pre-gap $\mathcal G=\set{(a_{\a},b_{\a}):\a<\k^+}$ is a \emph{gap} if there is no $c\sst\k$ such that $\abs{a_{\a}\setminus c}<\k$ and $\abs{b_{\a}\cap c}<\k$ for each $\a<\k^+$.

\begin{remark}\label{r:pregap}
	For any pre-gap $\set{(a_{\a},b_{\a}):\a<\k^+}$ in $\mathcal P(\k)$, and any $\a\le\b<\k^+$ we have $\abs{a_{\a}\cap b_{\b}}<\k$.
	To see this take $\a\le\b$ in $\k^+$.
	Since $a_{\b}\cap b_{\b}=\ps$ it must be that $a_{\a}\cap b_{\b}\sst a_{\a}\setminus a_{\b}$.
	By the definition of a pre-gap, we have $\abs{a_{\a}\setminus a_{\b}}<\k$, which then implies $\abs{a_{\a}\cap b_{\b}}<\k$.
\end{remark}

For $\b < \k^+$ and $\xi < \k$, set
$$
F^{\mathcal G}_{\xi}(\b) = \{\a < \b: a_{\a} \cap b_{\b} \sst \xi\}.
$$
In this case we are not able to show that sets $F^{\mathcal G}_{\xi}(\b)$ form a matrix.
However, our methods still give a topological group.
In this case, using these sets we get a subbase.
We proceed to explain this.
Same as before, for $\xi<\k$ and $\b<\k^+$ define $U^{\mathcal G}_{\xi}(\b) = \set{a \in [\k^+]^{<\o}: a \cap F^{\mathcal G}_{\xi}(\b) = \ps}$, and denote $$\textstyle\mathcal U_{\mathcal G}=\set{\bigcap_{(\xi,\b)\in K}U^{\mathcal G}_{\xi}(\b):K\in \left[\k\times\k^+\right]^{<\o}}.$$

\begin{theorem}
	\label{t:gap-group}
	The group $([\k^+]^{<\o}, \triangle)$, when equipped with $\cu_{\mathcal G}$ as the local base of the identity, becomes a topological group.
\end{theorem}

\begin{proof}
	As before, it is sufficent to show that $\cu_{\mathcal G}$ satisfies conditions (1)-(6) from Theorem~\ref{t:osnovna}.

	First we show (1).
	Take any $U \in \cu_{\mathcal G}$.
	Then there is $K \in [\k \times \k^+]^{<\o}$ such that $U = \bigcap_{(\xi,\b) \in K}U^{\mathcal G}_\xi(\b)$.
	Take any $x, y \in U$.
	Then $x \cap F^{\mathcal G}_\xi(\b) = \ps = y \cap F^{\mathcal G}_\xi(\b)$ for all $(\xi,\b) \in K$.
	This implies that for each $(\xi,\b) \in K$, $(x \cup y) \cap F^{\mathcal G}_\xi(\b) = \ps$.
	Since $x \triangle y \sst x \cup y$, we know that $(x\triangle y) \cap F^{\mathcal G}_\xi(\b) = \ps$ for each $(\xi,\b) \in K$.
	Hence, $x \triangle y \in \bigcap_{(\xi,\b) \in K}U^{\mathcal G}_\xi(\b) = U$, i.e. $U^2 \sst U$.
		
	Condition (2) is satisfied because $a=a^{-1}$ for $a$ in this group, condition (3) follows directly from (1), and condition (4) is true by the commutativity of $\triangle$.

	Now we prove (5).
	Take $U, V \in \cu_{\mathcal G}$.
	Then there are $K_1, K_2 \in [\k \times \k^+]^{<\o}$ such that
	$U = \bigcap_{(\xi,\b) \in K_1}U^{\mathcal G}_\xi(\b)$ and $V = \bigcap_{(\zeta,\a) \in K_2}U^{\mathcal G}_\zeta(\a)$.
	Consider the set $$\textstyle W = \bigcap_{(\eta, \g) \in K_1 \cup K_2}U^{\mathcal G}_\eta(\g).$$ 
	Clearly $W\in \cu_{\mathcal G}$.
	We will show that $W$ is a subset of $U \cap V$.	
	Let $x \in W$.
	Then $x \in U^{\mathcal G}_\eta(\g)$ for each $(\eta,\g) \in K_1 \cup K_2$.
	This implies that $x \cap F^{\mathcal G}_\eta(\g) = \ps$ for each $(\eta,\a) \in K_1 \cup K_2$.
	Then $x \cap F^{\mathcal G}_\eta(\g) = \ps$ for each $(\eta,\g) \in K_1$, and $x \cap F^{\mathcal G}_\eta(\g) =\ps$ for each $(\eta,\g) \in K_2$.
	Thus, $x \in \bigcap_{(\xi,\b) \in K_1}U^{\mathcal G}_\xi(\b) = U$ and $x \in \bigcap_{(\zeta,\a) \in K_2}U^{\mathcal G}_\zeta(\a) = V$.
	Then, $x \in U \cap V$, i.e. $W \sst U \cap V$, as required.

	Let us prove (6) now. Assume there is some $a \in [\k^+]^{<\o}\setminus\{\ps\}$ such that $a \in \bigcap\cu_{\mathcal G}$.
	Take $\a\in a$.
	Then for each $K \in [\k \times \k^+]^{<\o}$ we have $\a\notin F^{\mathcal G}_\xi(\b)$ for all $(\xi,\b) \in K$.
	This implies that $\a\notin F^{\mathcal G}_\xi(\b) = \ps$ for each $\xi < \k, \b < \k^+$.
	On the other hand, $\a < \a+1$ and by Remark \ref{r:pregap}, there is $\eta < \k$ such that $a_\a \cap b_{\a+1} \sst \eta$, i.e. $\a \in F^{\mathcal G}_\eta(\a+1)$ which is clearly contradicting the last observation.
\end{proof}

We will denote the topological group from the previous theorem by $G(\cu_{\mathcal G})$.

\begin{remark}\label{r:directedgap}
	Note that the proof of the previous theorem, in particular the proof of item (5) of Theorem \ref{t:osnovna}, gives that the family $\set{\bigcup_{(n,\b)\in K}F^{\mathcal G}_n(\b): K\in \left[\o\times\o_1\right]^{<\o}}$ is directed with respect to $\sst$.
\end{remark}

\begin{lemma}\label{l:gap-character}
	The group $G(\cu_\mathcal{G})$ has character $\k^+$.
\end{lemma}

\begin{proof}
	As before, it is enough to show there is no subset of $\cu_{\mathcal G}$ of size less then $\k^+$ forming a local base of $\ps$.
	Suppose the contrary, that $\mathcal K=\set{K_{\varepsilon}:\varepsilon<\a}$, for $\a\le\k$, is a family of finite subsets of $\k\times\k^+$ such that $\mathcal B=\set{\bigcap_{(\xi,\b)\in K_{\varepsilon}}U^{\mathcal G}_{\xi}(\b):\varepsilon<\k}$ is a local base of $\ps$.
	Consider the set $$B=\set{\b<\k^+:(\exists K\in\mathcal K)(\exists \xi<\k)\ (\xi,\b)\in K}.$$
	Since $\mathcal K$ is a collection of at most $\k$ many finite sets, its cardinality is at most $\k$.
	Hence $\abs{B}<\k^+$, and since $\k^+$ is regular, there is some $\g\in \k^+\setminus \sup(B)$.
	Then $\g\notin F^{\mathcal G}_{\xi}(\b)$ for any $\xi<\k$ and $\b\in B$, implying that \begin{equation}\label{eq:chargap}\textstyle(\forall\varepsilon<\k)\ \set{\g}\in \bigcap_{(\xi,\b)\in K_{\varepsilon}}U^{\mathcal G}_{\xi}(\b).\end{equation}
	Take any $\d\in \k^+\setminus (\g+1)$.
	Then, by Remark \ref{r:pregap} there is $\zeta<\k$ so that $a_{\g}\cap b_{\d}\sst \zeta$.
	This means that $\set{\g}$ does not belong to basic open $U^{\mathcal G}_{\zeta}(\d)$.
	Since $\mathcal B$ is a local base of $\ps$ there is $K_0\in\mathcal K$ so that $\bigcap_{(\xi,\b)\in K_0}U^{\mathcal G}_{\xi}(\b)\sst U^{\mathcal G}_{\zeta}(\d)$.
	Using (\ref{eq:chargap}), then $$\textstyle\set{\g}\in \bigcap_{(\xi,\b)\in K_0}U^{\mathcal G}_{\xi}(\b)\sst U^{\mathcal G}_{\zeta}(\d),$$ which is clearly in contradiction with $\set{\g}\notin U^{\mathcal G}_{\zeta}(\d)$.
\end{proof}

Again, as in the general theory of matrices, we are able to connect topological properties of groups with well-known critical properties of pre-gaps.

\begin{theorem}\label{t:gap-frechet}
	If the group $G(\cu_\mathcal{G})$ is of tightness at most $\k$, then $\mathcal{G}$ is a gap.
\end{theorem}

\begin{proof}
	Assume the contrary, that $\mathcal{G}$ is not a gap.
	Then there is a subset $c$ of $\k$ such that $\abs{a_{\a}\setminus c}<\k$ and $\abs{b_{\a} \cap c}<\k$ for each $\a < \k^+$.
	For $\xi<\k$ set
	$$
	A_{\xi} = \set{\a < \k^+: a_{\a}\setminus \xi \sst c \;\&\; b_{\a} \cap c \sst \xi}.
	$$
	Since $\k$ is regular and by the choice of $c$, we have $\k^+ = \bigcup_{\xi < \k}A_{\xi}$.
	Now, by regularity of $\k^+$, there is some $\eta < \k$ such that $A_{\eta}$ is cofinal in $\k^+$.
	Let us consider the family $\mathcal{A} = \set{ \set{\a}: \a \in A_{\eta} }$.
	
	First we will prove that $\ps \in \overline{\mathcal{A}}$.
	For this, we need to show that for each set $K\in \left[\k\times\k^+\right]^{<\o}$, the intersection of $\mathcal A$ and $\bigcap_{(\xi,\b)\in K}U^{\mathcal G}_{\xi}(\b)$ is nonempty.
	Take any $\a\in A_{\eta}$ such that $\a>\max\set{\b:(\exists \xi<\k)\ (\xi,\b)\in K}$.
	This is possible as $K$ is finite and $A_{\eta}$ is cofinal in $\k^+$.
	Then $\set{\a}\in \mathcal A$.
	By the choice of $\a$ we also know that $\a\notin F^{\mathcal G}_{\xi}(\b)$ for any $(\xi,\b)\in K$.
	Hence $\set{\a}\in U^{\mathcal G}_{\xi}(\b)$ for each $(\xi,\b)\in K$.
	Then $\set{\a}\in \mathcal A\cap \bigcap_{(\xi,\b)\in K}U^{\mathcal G}_{\xi}(\b)$ as required.
		
	Since the tightness of $G(\cu_\mathcal{G})$ is at most $\k$, there is a set $X\sst\mathcal{A}$ of size at most $\k$ so that $\ps\in\overline{X}$.
	By regularity of $\k^+$ and since $A_{\eta}$ is cofinal in $\k^+$, there is some $\b\in A_{\eta}$ such that $\a<\b$ whenever $\set{\a}\in X$.
	Let us prove that $\a\in F^{\mathcal G}_{\eta}(\b)$ whenever $\set{\a}\in X$.
	Suppose not, that $\set{\a}\in X\sst\mathcal A$ and that $\a\notin F^{\mathcal G}_{\eta}(\b)$.
	The latter assumption gives that $a_{\a}\cap b_{\b}\not\sst \eta$, i.e. there is $\zeta\in\k\setminus \eta$ such that $\zeta\in a_{\a}$ and $\zeta\in b_{\b}$.
	From $\zeta\ge\eta$, $\a\in A_{\eta}$ (since $\set{\a}\in X$), and $\zeta\in a_{\a}$, it follows that $\zeta\in c$.
	On the other hand, from $\zeta\ge\eta$, $\b\in A_{\eta}$, and $\zeta\in b_{\b}$, we have $\zeta\notin c$.
	Clearly, the last two observations are contradictory, so we proved that $\a\in F^{\mathcal G}_{\eta}(\b)$ whenever $\set{\a}\in X$.
	But this means that no element of $X$ belongs to the basic open set $U^{\mathcal G}_{\eta}(\b)$, contradicting the choice of $X$.
	Hence, the theorem is proved.
\end{proof}

Recall that an $(\o_1,\o_1)$ pre-gap $\mathcal G$ in $\cp(\o)$ is Hausdorff if the set $F^{\mathcal G}_n(\b)$ is finite for all $n<\o$ and $\b<\o_1$. In this case $\mathcal G$ has to be a gap.

\begin{theorem}
	If $\mathcal G$ is a Hausdorff gap in $\cp(\o)$, then $G(\cu_{\mathcal G})\equiv_T [\o_1]^{<\o}$.
\end{theorem}
\begin{proof}
	Note, as before, that the collection $\cu_{\mathcal G}$ ordered by $\supseteq$ is isomorphic as a poset to $\set{\bigcup_{(n,\b)\in K}F^{\mathcal G}_n(\b): K\in \left[\o\times\o_1\right]^{<\o}}$ ordered by $\sst$.
	This collection is uncountable, for example by Lemma \ref{l:gap-character}.
	Since $\mathcal G$ is Hausdorff, for each $K\in \left[\o\times\o_1\right]^{<\o}$ the set $\bigcup_{(n,\b)\in K}F^{\mathcal G}_n(\b)$ is a finite subset of $\o_1$.
	Hence, $G(\cu_{\mathcal G})\equiv_T [\o_1]^{<\o}$ by Remark \ref{r:types} and Remark \ref{r:directedgap}.
\end{proof}

\subsection*{Acknowledgements}
The first and the third author are partially supported by the grant from SFRS (7750027-SMART).
The first author is partially supported by the Ministry of Science, Technological Development and Innovation of the Republic of Serbia (Grants No. 451-03-137/2025-03/ 200125 \& 451-03-136/2025-03/ 200125)
The second author has been supported by the Ministry of Science, Technological Development and Innovation (Contract No. 451-03-137/2025-03/200156) and the Faculty of Technical Sciences, University of Novi Sad throught project “Scientific and Artistic Research Work of Researchers in Teaching and Associate Positions at the Faculty of Technical Sciences, University of Novi Sad 2025” (No. 01-50/295).
Also, the second author is grateful for the support within the project of the Department for general discipline in technology, Faculty of Technical Sciences under the title Soico-technological aspects of improving the teaching process in the English language in fundamental disciplines.
The third author is partially supported by grants from NSERC (455916) and CNRS (UMR7586).

\bibliographystyle{alpha}
\bibliography{bib_groups.bib}

\end{document}